\documentclass[10.5pt]{article}

\usepackage[margin=2.8cm]{geometry}
\usepackage{amsmath,amssymb,amsthm}
\usepackage{enumitem,color}
\usepackage{hyperref}
\usepackage{multirow,booktabs}
\usepackage{enumitem}

\usepackage{extarrows}
\usepackage{bm}
\usepackage{graphicx}
\usepackage{graphics}
\usepackage{caption2}
\usepackage{psfrag}
\usepackage{arydshln}

\usepackage{amscd}
\usepackage{amsfonts}
\usepackage{float}
\usepackage{latexsym}
\usepackage{color}
\usepackage{lscape}

\newtheorem{theorem}{Theorem}[section]
\newtheorem{lemma}[theorem]{Lemma}

\newtheorem{remark}{Remark}[section]
\newcommand{\om}{\Omega}

\newcommand{\p}{\partial}

\title{Spatial two-grid compact difference scheme for two-dimensional nonlinear diffusion-wave equations with variable exponent}

\author{Hao Zhang\thanks{School of Computer Science and Engineering, Sun Yat-sen University, Guangzhou 510006, Guangdong, P. R. China. (Email: zhangh925@mail2.sysu.edu.cn) }
\and
Kexin Li\thanks{School of Statistics and Mathematics, Yunnan University of Finance and Economics, Kunming 650221, Yunnan, P. R. China. (Email: likx1213@163.com) }
\and
Wenlin Qiu\thanks{School of Mathematics, Shandong University, Jinan 250100, Shandong, P. R. China. (Email: wlqiu@sdu.edu.cn) }
}

\date{}

\begin{document}

\maketitle

\begin{abstract}
This paper presents a spatial two-grid (STG) compact difference scheme for a two-dimensional (2D) nonlinear diffusion-wave equation with variable exponent, which describes, e.g., the propagation of mechanical diffusive waves in viscoelastic media with varying material properties. Following the idea of the convolution approach, the diffusion-wave model is first transformed into an equivalent formulation. A fully discrete scheme is then developed by applying a compact difference approximation in space and combining the averaged product integration rule with linear interpolation quadrature in time. An efficient high-order two-grid algorithm is constructed by solving a small-scale nonlinear system on the coarse grid and a large-scale linearized system on the fine grid, where the bicubic spline interpolation operator is used to project coarse-grid solutions to the fine grid. Under mild assumptions on the variable exponent $\alpha(t)$, the stability and convergence of the STG compact difference scheme are rigorously established. Numerical experiments are finally presented to verify the accuracy and efficiency of the proposed method.

\end{abstract}

\section{Introduction}

This study considers the following 2D nonlinear diffusion-wave model 
\begin{equation}\label{VtFDEs}
\begin{array}{c}
^{c}\p_{t}^{\alpha(t)}u(\bm x,t)-  \Delta u (\bm x,t)= f(u(\bm x,t)) ,\qquad (\bm x,t) \in \Omega\times(0,T],
\end{array}
\end{equation}
with the variable exponent $1<\alpha(t)<2$, which is subject to the initial conditions
\begin{equation}\label{ic}
u(\bm x,0)=u_0(\bm x),\quad\p_tu(\bm x,0)=\bar u_0(\bm x),\qquad \bm x\in \Omega,
\end{equation}
and Dirichlet boundary conditions
\begin{equation}\label{bc}
u(\bm x,t) = 0,\qquad (\bm x,t) \in \p \Omega\times[0,T].
\end{equation}
Here $\Omega \subset \mathbb{R}^2$ is a simply-connected bounded domain with the 
piecewise smooth boundary $\p \om$ with convex corners, $\bm x := (x,y)$ denotes the 
spatial variables. The 2D Laplacian is defined by 
$\Delta u := \partial_x^2 u + \partial_y^2 u$, and $f$, $u_0$, and $\bar{u}_0$ denote the source term and initial values, respectively. For the convenience of subsequent spatial discretization 
and analysis, we restrict the spatial domain to a rectangular area, that is $\Omega=(l_{x},r_{x})\times(l_{y},r_{y})$. The variable exponent fractional derivative of order $1<\alpha(t)<2$ is defined via \cite{LorHar}
\begin{equation}\nonumber
  ^{c}\p_{t}^{\alpha(t)}u(x,t):=(k*\p_t^2 u)(x,t)=\int_0^tk(t-s)\p_s^2u(x,s)ds,\qquad k(t):=\frac{t^{1-\alpha(t)}}{\Gamma(2-\alpha(t))}.
\end{equation}
Throughout this work, we assume that the nonlinear term $f(u)$ and $\alpha(t)$ satisfy the following conditions:
\begin{itemize}
    \item The nonlinear term $f(u)\in C^2(\mathbb{R})$ satisfies the globally boundedness condition $|f'(u)|\le Q_f$.
    \item The variable exponent $\alpha(t)$ is three times differentiable with $|\alpha'(t)|+|\alpha''(t)|+|\alpha'''(t)|\le Q_\alpha$ and $\alpha'(0)=0$.
\end{itemize}
Here, $Q_f$ and $Q_\alpha$ are positive constants. Moreover, we denote by $Q$ a generic constant, whose value may vary in different contexts.

For the constant-exponent case, i.e., $\alpha(t)\equiv \alpha$ with $1<\alpha<2$, models analogous to \eqref{VtFDEs}--\eqref{bc} have been employed to describe a variety of physical and engineering phenomena \cite{wulibeijing1}. In particular, it has been demonstrated in \cite{wulibeijing5,Lyu,wulibeijing4,wulibeijing3,wulibeijing2} that such models can simulate diverse processes, including anomalous diffusion, viscoelasticity, biological systems, financial dynamics, and quantum mechanics. Recent studies further indicate that the properties of materials or systems may not only exhibit anomalous behavior but also evolve over time in many dynamic processes \cite{SunZha,SunCha,wulibeijing6}. Motivated by these observations, variable-exponent models have attracted increasing attention in both theoretical analysis and numerical approximation, see, e.g., \cite{GarGiu,Hong,Jia,Liang,Ma,Zayernouri,ZenZhaKar,ZhuLiu}.

It is well recognized that fractional differential equations with variable exponents pose significant challenges \cite{DieFor}. Although variable-exponent fractional operators inherit the nonlocality and weak singularity of their constant-order counterparts, the lack of a convolution structure introduces additional difficulties in both mathematical analysis and numerical simulation. In recent years, a variety of accurate and efficient numerical methods have been developed for variable-exponent time- and space-fractional problems \cite{Cchen,LiWanWan,ZheWanJMAA,ZheWanSINUM,ZheWanIMA}. In contrast, rigorous analysis for diffusion-wave-type equations, such as model \eqref{VtFDEs}--\eqref{bc}, remains scarce. This scarcity arises from the fact that the variable-exponent Abel kernel $k(t)$ cannot be treated analytically by conventional techniques and may lack positive definiteness or monotonicity. Some researchers have addressed mathematical analysis and numerical approximation for variable-exponent fractional diffusion-wave models with an additional temporal leading term $\partial_t^2 u$, see, e.g., \cite{duruilian,zheWanCNSNS}. Nevertheless, to the best of the authors' knowledge, very few numerical methods exist for solving model \eqref{VtFDEs}--\eqref{bc} when $k * \partial_t^2 u$ serves as the leading term, while achieving second-order accuracy in time.

Recently, a convolution method was developed in \cite{Zhe}, which reformulates the variable-exponent subdiffusion problem into a more tractable form, facilitating rigorous analysis. Building on this approach, Qiu and Zheng \cite{Qiu25JSC} transformed the linear version of models \eqref{VtFDEs}--\eqref{bc} into integro-differential equations and investigated the well-posedness and regularity of their solutions. Motivated by these advances, we aim to construct appropriate numerical discretizations for the nonlinear diffusion-wave problem. Nonetheless, although certain ideas from \cite{Qiu25JSC,Zhe} can be adapted, the presence of the nonlinear source term introduces two major challenges in both the design of the numerical scheme and the accompanying theoretical analysis.

On the one hand, applying this approach to transform model \eqref{VtFDEs}--\eqref{bc} results in the original nonlinear term $f(u)$ becoming a nonlocal integral with a singular kernel. From the perspective of numerical analysis, a primary difficulty is that this transformed nonlinear term may hinder the construction of high-order schemes. To address this issue, we employ the averaged product integration (PI) rule proposed by McLean and Mustapha \cite{Mclean} to approximate the nonlinear term in the temporal direction.

On the other hand, compared with the linear case, the computational cost of solving nonlinear problems increases dramatically as the mesh is refined. To improve efficiency, Xu \cite{Xujinchao1,Xujinchao2} proposed the two-grid algorithm, which reduces computational effort without sacrificing accuracy. Following this idea, several researchers have applied this technique to nonlinear problems, see, e.g., \cite{twogird1,twogird2,twogird3}. However, for the STG difference method, the finite difference approach produces discrete solutions on the grid, making it challenging to preserve spatial accuracy when transferring the coarse-grid solution to the fine grid. Specifically, the lack of a rigorous analysis for an accuracy-preserving mapping operator has limited the development of high-order STG algorithms in the literature. Recently, Fu et al. \cite{fuhongfei} proposed and analyzed a high-order mapping operator between two grids, enabling the construction and analysis of high-order STG schemes for nonlinear problems. In this work, we employ this high-order mapping operator to develop an STG compact difference scheme for solving model \eqref{VtFDEs}--\eqref{bc}.

The main contributions of this work can be summarized as follows:

\textbf{({\romannumeral1})} Building on \cite{Qiu25JSC,Zhe}, we transform the nonlinear variable-exponent diffusion-wave model into a nonlinear integro-differential formulation. By applying the averaged PI rule to the diffusion and nonlinear terms and interpolating the kernel $k(t)$ as in \cite{ADIarxiv}, spatial discretization via a compact difference method combined with the STG approach yields an efficient two-grid fourth-order scheme. This framework allows rigorous analysis under mild assumptions on a linear $\alpha(t)$, without requiring $\alpha''(0)=0$ as in \cite{Qiu25JSC}, and extends the STG method from constant-exponent models of \cite{fuhongfei} to the variable-exponent setting.

\textbf{({\romannumeral2})} Leveraging discrete energy techniques, we rigorously establish the stability and convergence of the proposed STG compact difference scheme. Notably, under suitable regularity assumptions, the scheme achieves second-order accuracy in time, free from order reduction induced by variable exponents, and fourth-order accuracy in space. Numerical experiments confirm its superior efficiency and effectiveness compared with fully nonlinear counterparts.

\textbf{({\romannumeral3})}  Under the same regularity assumptions (see Remark \ref{regularity}) with $\alpha(0) \in (1,2)$, we perform numerical analysis on model \eqref{VtFDEs} after applying the convolution method, attaining the same high-order accuracy as the scheme in \cite{Lyu} on a uniform grid, without the stringent grid ratio required for a non-uniform grid.

The remainder of this paper is organized as follows. Section \ref{sec2} presents preliminaries on temporal and spatial discretizations. In Section \ref{sec3}, we develop the standard nonlinear compact difference scheme and the STG compact difference scheme. The stability and convergence of the proposed STG scheme are established in Section \ref{sec4}. Finally, Section \ref{sec5} reports numerical experiments that validate the accuracy and theoretical reliability of the scheme.

\section{Preliminaries}\label{sec2}
In this section, we introduce notations and preliminary lemmas that will be used in the subsequent sections.

\subsection{The spatial preliminaries}
Let spatial step sizes $\kappa_{x}=(r_{x}-l_{x})/{M_{\kappa,x}}$ and 
$\kappa_{y}=(r_{y}-l_{y})/{M_{\kappa,y}}$ for $M_{\kappa,x}, M_{\kappa,y}\in \mathbb{Z}^{+}$ 
with $\kappa=H$ or $h$. For a given positive integer $J\ge 2$, we assume that 
$M_{h,x}=J M_{H,x}$ and $M_{h,y}=J M_{H,y}$. Then, we can perform coarse grid partition 
of $\Omega$ by $x_{i}=l_{x}+iH_{x}$ for $0\le i\le M_{H,x}$ and $y_{j}=l_{y}+jH_{y}$ 
for $0\le j\le M_{H,y}$. Similarly, the fine grid partition of $\Omega$ is given by 
$x_{i}=l_{x}+ih_{x}$ for $0\le i\le M_{h,x}$ and $y_{j}=l_{y}+jh_{y}$ for 
$0\le j\le M_{h,y}$. Here, we introduce the spaces of grid functions:

\begin{equation*}
\mathfrak{U}_\kappa=\{u \mid u=\{u_{i,j}\}, (i,j)\in \bar{\omega}_\kappa\},\qquad 
\dot{\mathfrak{U}}_\kappa=\{u\mid u\in\mathfrak{U}_\kappa~\text{and}~u_{i,j}=0~\text{if}~(i,j)\in\partial\omega_\kappa\},
\end{equation*}
with the notations 
$\omega_{\kappa}=\{(i,j)\:|\:1 \leq i \leq M_{\kappa,x}-1, 1 \leq j \leq M_{\kappa,y}-1\}$, 
$\bar{\omega}_{\kappa}= \{(i,j)\:|\:0\leq i \leq M_{\kappa,x}, 0 \leq j \leq M_{\kappa,y}\}$
and $\partial\omega_{\kappa}=\bar{\omega}_{\kappa}\setminus\omega_{\kappa}$. 
For any grid function $u\in \mathfrak{U}_{\kappa}$, define the following discrete operators
\begin{align*}
\mathcal{A}_{\kappa,x}u_{i,j}=
\left\{\begin{array}{ccc}
\frac{u_{i-1,j}+10u_{i,j}+u_{i+1,j}}{12},& 1\leq i\leq M_{\kappa,x}-1,&\quad0\leq j\leq M_{\kappa,y},\\
u_{i,j},& \quad i=0,M_{\kappa,x},&\quad0\leq j\leq M_{\kappa,y},
\end{array}\right.\\
\mathcal{A}_{\kappa,y}u_{i,j}=
\left\{\begin{array}{ccc}
\frac{u_{i,j-1}+10u_{i,j}+u_{i,j+1}}{12},& 1\leq j\leq M_{\kappa,y}-1,&\quad0\leq i\leq M_{\kappa,x},\\
u_{i,j},& \quad j=0,M_{\kappa,y},&\quad0\leq i\leq M_{\kappa,x},
\end{array}\right.
\end{align*}
and
\begin{align*}
&\delta_{\kappa,x}^{2}u_{i,j}=\frac{u_{i+1,j}-2u_{i,j}+u_{i-1,j}}{\kappa_{x}^{2}},\qquad 
\delta_{\kappa,y}^{2}u_{i,j}=\frac{u_{i,j+1}-2u_{i,j}+u_{i,j-1}}{\kappa_{y}^{2}},\\
&\mathcal{A}_{\kappa}u_{i,j}=\mathcal{A}_{\kappa,x}\mathcal{A}_{\kappa,y}u_{i,j},\qquad\qquad\qquad \Lambda_{\kappa}u_{i,j}=\mathcal{A}_{\kappa,y}\delta_{\kappa,x}^{2}u_{i,j}+\mathcal{A}_{\kappa,x}\delta_{\kappa,y}^{2}u_{i,j}.
\end{align*}
Besides, for any grid functions $u$ and $w$ in the space $\dot{\mathfrak{U}}_{\kappa}$, invoke the discrete inner products and norms
\begin{equation*}
\begin{aligned}
&(u,w)_{\kappa}=\kappa_{x}\kappa_{y}\sum_{i=1}^{M_{\kappa,x}-1}\sum_{j=1}^{M_{\kappa,y}-1}u_{i,j}w_{i,j}, \qquad\qquad\quad \|u\|_{\kappa}=\sqrt{(u,u)_{\kappa}},\\
&(u,w)_{\mathcal{A}_{\kappa}}=(u,\mathcal{A}_{\kappa}w)_{\kappa},\quad 
\|u\|_{\mathcal{A}_{\kappa}}=\sqrt{(u,\mathcal{A}_{\kappa}u)_{\kappa}},\quad
\|u\|_{\kappa,\infty}=\max\limits_{0\leq i\leq M_{\kappa,x},0\leq j\leq M_{\kappa,y}}\{|u_{i,j}|\}.
\end{aligned}
\end{equation*}
We next present several auxiliary lemmas.

\begin{lemma}\label{space1}
\cite[Lemma 4.1]{duruilian}
Let $w\in\dot{\mathfrak{U}}_{\kappa}$, then it holds that
\begin{equation*}
\frac{1}{3}\|w\|^{2}_{\kappa}\leq\|w\|_{\mathcal A_{\kappa}}^{2}\leq\|w\|_{\kappa}^{2},\qquad   
\|\mathcal{A}_{\kappa}w\|_{\kappa}\leq\|w\|_{\kappa},\qquad \|\mathcal{A}_{\kappa}w\|_{\kappa,\infty}\leq\|w\|_{\kappa,\infty}.
\end{equation*}
\end{lemma}

\begin{lemma}\label{space2}
\cite{sunzhizhong1}
Suppose the function $G(x,y)\in C_{x,y}^{6,6}([l_{x},r_{x}]\times[l_{y},r_{y}])$, we have
\begin{align*}
&\frac{1}{12}\left[\partial_{x}^{2}G(x_{i-1},y_{j})+10\partial_{x}^{2}G(x_{i},y_{j})+\partial_{x}^{2}G(x_{i+1},y_{j})\right]-\frac{\kappa_{x}^{4}}{240}\partial_{x}^{6}G(\hat{\theta_{i}},y_{j})\\
=&\frac{1}{\kappa_{x}^{2}}\left[G(x_{i-1},y_{j})-2G(x_{i},y_{j})+G(x_{i+1},y_{j})\right],\quad 1\le i\le M_{\kappa,x}-1, \quad 0\le j\le M_{\kappa,y},\\
&\frac{1}{12}\left[\partial_{y}^{2}G(x_{i},y_{j-1})+10\partial_{y}^{2}G(x_{i},y_{j})+\partial_{y}^{2}G(x_{i},y_{j+1})\right]-\frac{\kappa_{y}^{4}}{240}\partial_{y}^{6}G(x_{i},\tilde{\theta}_{j})\\
=&\frac{1}{\kappa_{y}^{2}}\left[G(x_{i},y_{j-1})-2G(x_{i},y_{j})+G(x_{i},y_{j+1})\right],\quad 0\le i\le M_{\kappa,x}, \quad 1\le j\le M_{\kappa,y}-1.
\end{align*}
with $\hat{\theta}_i\in(x_{i-1},x_{i+1})$ and $\tilde{\theta}_j\in(y_{j-1},y_{j+1})$.
\end{lemma}

\begin{lemma}\label{positivityofLambda}
    \cite[Lemma 5.1]{wangyuanming}
    For any $u,w\in\dot{\mathfrak{U}}_\kappa$, there exists a spatial operator $\nabla_\kappa^*$ such that 
    \begin{equation*}
        -\left(\Lambda_\kappa u,w\right)_\kappa=\left(\nabla_\kappa^*u,\nabla_\kappa^*w\right)_\kappa.
    \end{equation*}
\end{lemma}

\subsection{The temporal preliminaries}
For given $N\in\mathbb{Z}^+$, we perform a uniform partition of the time interval $[0,T]$
by $t_{n}=n\tau$ for $0\le n\le N$ with $\tau=T/N$. Then, denote
\begin{equation*}
t_{n-\frac{1}{2}}=\frac{t_{n}+t_{n-1}}{2},\quad 1\le n\le N,
\end{equation*}
and $v^{n}=v(t_{n})$ for $0\le n\le N$. Let $\mathcal{G}=\{v \mid v=\{v^n\},~1\le n\le N\}$
be a temporal grid function space. For $v\in \mathcal{G}$, define the notations as
\begin{equation*}
v^{n-\frac{1}{2}}=\frac{1}{2}\left(v^{n}+v^{n-1}\right),\quad \delta_{t}v^{n}=\frac{v^{n}-v^{n-1}}{\tau},\quad 1\le n\le N.
\end{equation*}

In the next section, applying the convolution method to problem \eqref{VtFDEs}--\eqref{bc} leads to the following nonlocal integral in the transformed model
\begin{equation}\label{e2.1}
  \frac{1}{\tau}\int_{t_{n-1}}^{t_{n}}\int_{0}^{t}\beta_{\bar{\alpha}}(t-s)\phi(s)dsdt,\quad \text{for}~\beta_{\bar{\alpha}}(t)=\frac{t^{\bar{\alpha}-1}}{\Gamma(\bar{\alpha})},\quad 1\le n\le N,
\end{equation}
with $\bar{\alpha}\in(0,1)$ being a constant exponent. In order to numerically approximate this type of integral, we introduce the averaged PI rule \cite{Mclean}
\begin{equation*}
\mathfrak{I}^{n}_{\bar{\alpha}}\phi^{n}=\frac{1}{\tau}\int_{t_{n-1}}^{t_{n}}\int_{0}^{t}\beta_{\bar{\alpha}}(t-s)\breve{\phi}(s)dsdt=\lambda_{n,1}\phi^{1}+\sum_{j=2}^{n}\lambda_{n,j}\phi^{j-\frac{1}{2}},
\end{equation*}
in which
\begin{equation}\label{consappro}
\breve{\phi}(s)=
\begin{cases}
  \phi^1,&0<s<\tau,\\
  \phi^{n-\frac{1}{2}},&t_{n-1}<s<t_n,\quad n\ge2,
\end{cases}
\end{equation}
with the weights
\begin{equation*}
  \lambda_{n,j}=\frac{1}{\tau}\int_{t_{n-1}}^{t_n}\int_{t_{j-1}}^{\min\{t,t_j\}}\beta_{\bar{\alpha}}(t-s)dsdt>0,\quad1\leq j\leq n.
\end{equation*}
We then have 
\begin{align}
&\lambda_{n,1}=\frac{1}{\tau}\int_{t_{n-1}}^{t_n}\int_{0}^{\tau}\beta_{\bar{\alpha}}(t-s)dsdt\le \frac{1}{\tau}\int_{t_{n-1}}^{t_n}\frac{t^{\bar{\alpha}}}{\Gamma(\bar{\alpha}+1)}dt\le Q_{\bar{\alpha},T},\qquad n\ge 2,\label{lam1}\\
&\lambda_{n,n}=\frac{1}{\tau}\int_{t_{n-1}}^{t_n}\int_{t_{n-1}}^{t}\beta_{\bar{\alpha}}(t-s)dsdt= \frac{\tau^{\bar{\alpha}}}{\Gamma(\bar{\alpha}+2)}\le Q_{\bar{\alpha},T},\qquad n\ge 1,\label{lam2} 
\end{align}
and 
\begin{equation}\label{lam3}
\begin{aligned}
\sum_{j=1}^{n}\lambda_{n,j}&=\frac{1}{\tau}\sum_{j=1}^{n}\int_{t_{n-1}}^{t_n}\int_{t_{j-1}}^{\min\{t,t_j\}}\beta_{\bar{\alpha}}(t-s)dsdt\\
&=\frac{1}{\tau}\int_{t_{n-1}}^{t_n}\frac{t^{\bar{\alpha}}}{\Gamma(\bar{\alpha}+1)}dt\le Q_{\bar{\alpha},T},\qquad n\ge 1.
\end{aligned}
\end{equation}
Here, the constant $Q_{\bar{\alpha},T}$ denotes a positive constant depending only on $\bar{\alpha}$ and $T$. 
In what follows, similar notations will be used without further explanation.

Furthermore, since the local truncation error between \eqref{e2.1} and 
$\mathfrak{I}^{n}_{\bar{\alpha}}\phi^{n}$ depends on the regularity of the time-dependent function 
$\phi$, this issue will be discussed in detail in a later section. Next, we present an important lemma related to this property.

\begin{lemma}\label{PIpositive}\cite{wangyuanming}
For any grid function $v\in \mathcal{G}$ and $\alpha\in (0,1)$, it holds that
\begin{equation*}
v^{1}\mathfrak{I}_{\alpha}^{1}v^{1}+\sum_{n=2}^{N}v^{n-\frac{1}{2}}\mathfrak{I}_{\alpha}^{n}v^{n}\ge 0.
\end{equation*}
\end{lemma}

\section{Establishment of two compact difference schemes}\label{sec3}

In this section, we develop the standard nonlinear compact difference scheme for problem \eqref{VtFDEs}--\eqref{bc} and then construct the corresponding STG compact difference scheme.

\subsection{Model reformulation}
Before that, we first transform the variable-exponent diffusion-wave model into a more tractable formulation. At this stage, we introduce the generalized identity function
\begin{equation}\label{mh7}
  g(t)=(\beta_{\alpha_0-1}*k)(t)=\int_0^t\frac{(t-s)^{\alpha_0-2}}{\Gamma(\alpha_0-1)}\frac{s^{1-\alpha(s)}}{\Gamma(2-\alpha(s))}ds,
\end{equation}
where $\alpha_0=\alpha(0)$. It is shown in \cite{Zhe} that 
$g(0)=1$ and
\begin{align}\label{g}
|g|\leq Q_g,~~|g'|\leq Q_g(|\ln t|+1),~~|g^{(m)}|\leq Q_g t^{-(m-1)},~~m=2,3\text{ for }t\in (0,T].
\end{align}
We take the convolution of (\ref{VtFDEs}) and $\beta_{\alpha_0-1}$ to arrive at
\begin{equation*}
\beta_{\alpha_0-1}*[(k*\p_t^2 u)-\Delta u - f(u)]=0,
\end{equation*}
which follows \eqref{mh7} to get
\begin{equation*}
g*\p_t^2u -\beta_{\alpha_0-1}*  \Delta u=\beta_{\alpha_0-1}*f(u).
\end{equation*}
Denoting $\bar{\alpha} = \alpha_0 -1$, in combination with integration by parts, 
we obtain the transformed model
\begin{equation}\label{Model2}
\begin{aligned}
&\partial_t u(\bm x,t)+ \int_0^t g'(t-s) \p_s u(\bm x,s)ds - \int_0^t \beta_{\bar{\alpha}}(t-s) \Delta u(\bm x,s)ds\\
=&\int_0^t \beta_{\bar{\alpha}}(t-s) f(u(\bm x,s))ds + g(t)\bar{u}_0(\bm x),\quad (\bm x,t) \in \Omega\times(0,T],
\end{aligned}
\end{equation}
equipped with the initial and boundary conditions \eqref{ic}-\eqref{bc}. From 
\cite{Qiu25JSC}, we know that the transformed model and 
\eqref{VtFDEs}-\eqref{bc} are indeed equivalent, which indicates that attention to 
\eqref{VtFDEs} can be diverted to \eqref{Model2}. Thus, we below consider the 
numerical approximations of \eqref{Model2}. Under this circumstance, some improved properties on $g(t)$ compared with those in \eqref{g} are needed.

\begin{lemma}\label{gpro}\cite{ADIarxiv}
Let $g(t)$ be denoted by \eqref{mh7}. If $\alpha'(0)=0$, then $|g'(t)|\leq Q_g$ for $t\in [0,T]$ and $|g''(t)|\leq Q_g(|\ln t|+1)$ for $t\in (0,T]$.
\end{lemma}

The improved properties of the generalized identity function $g$ facilitate the subsequent numerical analysis. Moreover, the following remark on the regularity of $u$ plays a key role in the proofs of the main results.

\begin{remark}\label{regularity}
Based on the investigations on the regularity of linear and nonlinear diffusion-wave
model with variable exponent \cite{duruilian,Qiu25JSC}, it is reasonable 
to impose the following regularity assumptions:\\
\textbf{(A1)} $\p_tu$, $\p_t\p_x^2u$, $\p_t\p_y^2u$, $\p_x^6$ and $\p_y^6u$ are 
continuous in $\bar{\Omega}\times[0,T]$ with $\bar{\Omega}=\Omega\cup\partial\Omega$, 
and there exist a positive constant $Q$ such that
\begin{equation*}
\left|\partial_tu(\cdot,t)\right|\leq Qt^{\alpha_0-1},\quad \left|\partial_{\chi}^6 u(\cdot,t)\right|\leq Q,\quad\left|\partial_t\partial_\chi^2u(\cdot,t)\right|\leq Ct^{\alpha_0-1},\quad (\chi=x,y),
\end{equation*}
for all $(\bm{x},t)\in\bar{\Omega}\times (0,T]$.\\
\textbf{(A2)} $\p_t^2u$, $\p_t^3u$, $\p_t^2\p_x^2u$ and $\p_t^2\p_y^2u$ are continuous 
in $\bar{\Omega}\times(0,T]$, and there exist a positive constant $Q$ such that
\begin{equation*}
\left|\partial_t^2u(\cdot,t)\right|\leq Qt^{\alpha_0-2},\quad\left|\partial_t^3u(\cdot,t)\right|\leq Qt^{\alpha_0-3},\quad\left|\partial_t^2\partial_\chi^2u(\cdot,t)\right|\leq Qt^{\alpha_0-2},\quad (\chi=x,y),
\end{equation*}
for all $(\bm{x},t)\in\bar{\Omega}\times(0,T]$.
\end{remark}

\subsection{Construction of standard compact difference scheme}
For brevity, we denote by $u$ the solution of problem \eqref{VtFDEs}--\eqref{bc}, and define 
\begin{equation}\nonumber
u_{i,j}^{n}=u(x_{i},y_{j},t_{n}),\quad b_{i,j}^{n}=\frac{1}{\tau}\int_{t_{n-1}}^{t_{n}}g(t)\bar{u}_{0}(x_{i},y_{j})dt,\qquad (i,j)\in\bar{\omega}_{\kappa},~~1\le n\le N. 
\end{equation}
Considering \eqref{Model2} at at the point $(x_{i},y_{j})$, we integrate it for 
$t$ from $t_{n-1}$ to $t_{n}$ and multiply $1/\tau$ to obtain
\begin{equation}\label{e3.4}
\begin{aligned}
&\delta_{t}u_{i,j}^{n}+\frac{1}{\tau}\int_{t_{n-1}}^{t_{n}}\widehat{G}(x_{i},y_{j},t)dt-\frac{1}{\tau}\int_{t_{n-1}}^{t_{n}}\int_{0}^{t}\beta_{\tilde{\alpha}}(t-s)\Delta u(x_{i},y_{j},s)dsdt\\
=&\frac{1}{\tau}\int_{t_{n-1}}^{t_{n}}\int_0^t \beta_{\bar{\alpha}}(t-s) f(u(x_{i},y_{j},s))dsdt+b_{i,j}^{n},\quad (i,j)\in\omega_{\kappa},~~1\le n\le N,
\end{aligned}
\end{equation}
where 
\begin{equation*}
  \widehat{G}(x,y,t)=\int_0^tg'(t-s)\partial_s u(x,y,s)ds,\quad t\geq0.
\end{equation*}
Based on this notation, it is easy to get 
\begin{equation*}
\begin{aligned}
&\partial_{t}\widehat{G}(x,y,t)=g^{\prime}(0)\partial_{t}u(x,y,t)+\int_{0}^{t}g^{\prime\prime}(t-s)\partial_{s}u(x,y,s)ds,\\
&\partial_{t}^{2}\widehat{G}(x,y,t)=g^{\prime}(0)\partial_{t}^{2}u(x,y,t)+g^{\prime\prime}(t)\bar{u}_{0}(x,y)+\int_{0}^{t}g^{\prime\prime}(s)\partial_{s}^{2}u(x,y,t-s)ds,
\end{aligned}
\end{equation*}
which, together with Lemmas \ref{gpro} and the regularity conditions in Remark \ref{regularity}, gives 
\begin{equation}\label{Gpro}
    |\partial_{t}^{2}\widehat{G}(x,y,t)|\le Q\left(t^{\alpha_0-2}+t^{-\varepsilon}+t^{\alpha_0-1-\varepsilon}\right)\le Qt^{\alpha_0-2},\qquad 0< \varepsilon\ll 1,\qquad t\to0^{+}.
\end{equation}

Next, we apply various discretization techniques to approximate the terms in \eqref{e3.4}. In particular, piecewise linear interpolation is used to obtain
\begin{align}\nonumber
  \widehat{G}(x_{i},y_{j},t_n)= \sum\limits_{k=1}^{n} w_{n-k} \delta_{t}u_{i,j}^{k} + (R_1)_{i,j}^n, \quad (i,j)\in\omega_{\kappa},~~1\le n\le N,
\end{align}
with $w_k=g(t_{k+1})-g(t_{k})$, and from \cite{Qiu25JSC}, it follows that
\begin{equation}\label{R1}
\begin{split}
(R_1)_{i,j}^n = \sum\limits_{k=1}^{n}    \int_{t_{k-1}}^{t_k} \Biggl[ & \frac{t_{k-1}-s}{\tau} \int_{s}^{t_k} g''(t_n-\theta)(t_k-\theta)d\theta \\
& + \frac{s-t_k}{\tau} \int_{s}^{t_{k-1}} g''(t_n-\theta)(t_{k-1}-\theta)d\theta   \Biggl] \p_{s}^2 u(x_{i},y_j,s)ds.
\end{split}
\end{equation}
We then combine the middle rectangle formula to discretize the second left-hand 
side term as
\begin{equation}\label{e3.6}
\begin{aligned}
\frac{1}{\tau}\int_{t_{n-1}}^{t_{n}}\widehat{G}(x_{i},y_{j},t)dt& =\frac{\hat{G}(x_{i},y_{j},t_{n})+\hat{G}(x_{i},y_{j},t_{n-1})}{2}+(R_{2})_{i,j}^{n}  \\
&=\sum_{k=1}^{n}\tilde{w}_{n-k}\delta_{t}u_{i,j}^{k}+(R_{1})_{i,j}^{n-\frac{1}{2}}+(R_{2})_{i,j}^{n},
\end{aligned}
\end{equation}
with 
\begin{equation*}
\tilde{w}_{0}=\frac{w_{0}}{2},\qquad \tilde{w}_{k}=\frac{1}{2}\left(w_{k}+w_{k-1}\right),\quad 1\le k\le n-1,
\end{equation*}
and 
\begin{equation}\label{R2}
  (R_2)_{i,j}^n=\frac{1}{\tau}\int_{t_{n-1}}^{t_n}\widehat G(x_{i},y_{j},t)dt-\frac{\widehat G(x_{i},y_{j},t_n)+\widehat G(x_{i},y_{j},t_{n-1})}{2}.
\end{equation}
For the third term on the left-hand side, we apply the averaged PI rule to get 
\begin{equation}\label{e3.8}
\begin{aligned}
\frac{1}{\tau}\int_{t_{n-1}}^{t_{n}}\int_{0}^{t}\beta_{\bar{\alpha}}(t-s)\Delta u(x_{i},y_j,s)dsdt=&\frac{1}{\tau}\int_{t_{n-1}}^{t_{n}}\int_{0}^{t}\beta_{\bar{\alpha}}(t-s)\Delta \breve{u}(x_{i},y_j,s)dsdt+(R_{3})_{i,j}^{n}\\
=&\mathfrak{I}_{\bar{\alpha}}^{n}(\Delta u)_{i,j}^{n}+(R_{3})_{i,j}^n,
\end{aligned}
\end{equation}
in which $\breve{u}(\cdot,s)$ is similar to that in \eqref{consappro}, and 
\begin{equation}\label{R3}
(R_3)_{i,j}^n=\frac{1}{\tau}\int_{t_{n-1}}^{t_n}\int_0^t\beta_{\bar{\alpha}}(t-s)\left[\Delta u(x_{i},y_{j},s)-\Delta\breve{u}(x_{i},y_{j},s)\right]dsdt.
\end{equation}
Regarding the nonlinear term $f(u)$, we also introduce the notations 
\begin{equation}\nonumber
f_{i,j}^{n}=f(u(x_{i},y_{j},t_{n})),\quad f_{i,j}^{n-\frac{1}{2}}=\frac{f_{i,j}^{n}+f_{i,j}^{n-1}}{2},\quad (i,j)\in\omega_{\kappa},~~1\le n\le N. 
\end{equation}
In this case, another application of PI rule leads to
\begin{equation}\label{e3.10}
\begin{aligned}
\frac{1}{\tau}\int_{t_{n-1}}^{t_{n}}\int_0^t \beta_{\bar{\alpha}}(t-s) f(u(x_{i},y_{j},s))dsdt
=&\frac{1}{\tau}\int_{t_{n-1}}^{t_{n}}\int_0^t \beta_{\bar{\alpha}}(t-s) \breve{f}(u(x_{i},y_{j},s))dsdt+(R_{4})_{i,j}^{n}\\
=&\mathfrak{I}_{\bar{\alpha}}^{n}f_{i,j}^{n}+(R_{4})_{i,j}^n.
\end{aligned}
\end{equation}
Then, substituting \eqref{e3.6}, \eqref{e3.8} and \eqref{e3.10} into \eqref{e3.4}, we 
have 
\begin{equation}\label{e3.11}
\delta_{t}u_{i,j}^{n}+\sum_{k=1}^{n}\tilde{w}_{n-k}\delta_{t}u_{i,j}^{k}-\mathfrak{I}_{\bar{\alpha}}^{n}(\Delta u_{i,j})^{n}=\mathfrak{I}_{\bar{\alpha}}^{n}f_{i,j}^{n}+b_{i,j}^{n}+(R_{5})_{i,j}^n,
\end{equation}
where
\begin{equation*}
(R_{5})_{i,j}^n=-(R_{1})_{i,j}^{n-\frac{1}{2}}-(R_{2})_{i,j}^{n}+(R_{3})_{i,j}^{n}+(R_{4})_{i,j}^{n},\quad \text{for}~~(i,j)\in\omega_{\kappa},~~1\le n\le N.
\end{equation*}

At this point, we proceed to consider the spatial approximation of \eqref{e3.11}. Obviously, 
Lemma \ref{space2} indicates that 
\begin{equation*}
[\mathcal{A}_{\kappa,x}\p_{x}^{2} u]_{i,j}^n=[\delta_{\kappa,x}^2u]_{i,j}^n+(R_x)_{i,j}^n,\quad[\mathcal{A}_{\kappa,y}\p_{y}^{2}u]_{i,j}^n=[\delta_{\kappa,y}^{2}u]_{i,j}^n+(R_y)_{i,j}^n,
\end{equation*}
with $|(R_x)_{i,j}^n|\le Q\kappa_{x}^{4}$ and $|(R_y)_{i,j}^n|\le Q\kappa_{y}^{4}$ under the condtions in Remark \ref{regularity}. Thus, performing the compact difference operator 
$\mathcal{A}_{\kappa}$ on both sides of \eqref{e3.11}, we obtain
\begin{equation}\label{e3.12}
  [\mathcal{A}_{\kappa}\delta_{t}u]_{i,j}^{n}+\sum_{k=1}^{n}\tilde{w}_{n-k}[\mathcal{A}_{\kappa}\delta_{t}u]_{i,j}^{k}-\mathfrak{I}_{\bar{\alpha}}^{n}[\Lambda_{\kappa} u]_{i,j}^{n}=\mathfrak{I}_{\bar{\alpha}}^{n}[\mathcal{A}_{\kappa}f]_{i,j}^{n}+[\mathcal{A}_{\kappa}b]_{i,j}^{n}+(R)_{i,j}^n,
\end{equation}
and 
\begin{equation*}
(R)_{i,j}^{n}=\mathcal{A}_{\kappa}(R_{5})_{i,j}^{n}+\mathcal{A}_{\kappa,y}\mathfrak{I}_{\bar{\alpha}}^{n}(R_x)_{i,j}^n+\mathcal{A}_{\kappa,x}\mathfrak{I}_{\bar{\alpha}}^{n}(R_y)_{i,j}^n,\quad (i,j)\in\omega_{\kappa},~~1\le n\le N. 
\end{equation*}

For the convenience of subsequent analysis, we derive the bounds of $(R)_{i}^{n}$ 
in the following lemma.

\begin{lemma}\label{errorestimate}
On the basis of the regularity assumptions, we derive the following estimate for the truncation error
\begin{equation*}
\begin{aligned}
&\tau\sum_{m=1}^n\|(R)^m\|_{h,\infty}\leq Q(\tau^2+h_x^4+h_y^4),\quad~~1\leq n\leq N,\\
&\tau\sum_{m=1}^n\|(R)^m\|_{H,\infty}\leq Q(\tau^2+H_x^4+H_y^4),\quad1\leq n\leq N,
\end{aligned}
\end{equation*}
in which $Q=Q(T,\bar{\alpha},Q_g,Q_f)$.
\end{lemma}
\begin{proof}
We first obtain 
\begin{equation}\label{e3.13}
\begin{aligned}
\tau\sum_{m=1}^n\|(R)^m\|_{h,\infty}\le \tau\sum_{m=1}^n\big(&\|\mathcal{A}_{h}(R_{1})^{m-\frac{1}{2}}\|_{h,\infty}
+\|\mathcal{A}_{h}(R_{2})^{m}\|_{h,\infty}
+\|\mathcal{A}_{h}(R_{3})^{m}\|_{h,\infty}
+\|\mathcal{A}_{h}(R_{4})^{m}\|_{h,\infty}\\
&+\mathfrak{I}_{\bar{\alpha}}^{m}\|\mathcal{A}_{h,y}(R_x)^m\|_{h,\infty}
+\mathfrak{I}_{\bar{\alpha}}^{m}\|\mathcal{A}_{h,x}(R_y)^m\|_{h,\infty} \big).\\
\end{aligned}
\end{equation}
Each of \eqref{e3.13} will be estimated. By \eqref{R1}, we split $(R_1)_{i,j}^m=(R_{11})_{i,j}^m+(R_{12})_{i,j}^m$, and
\begin{align*}
    (R_{11})_{i,j}^m &= \sum\limits_{k=1}^{m}    \int_{t_{k-1}}^{t_k} \Biggl[  \frac{t_{k-1}-s}{\tau} \int_{s}^{t_k} g''(t_m-\theta)(t_k-\theta)d\theta\Biggl]\p_{s}^2 u(x_i,y_j,s)ds,\\
    (R_{12})_{i,j}^m&=\sum\limits_{k=1}^{m}\int_{t_{k-1}}^{t_k} \Biggl[\frac{s-t_k}{\tau} \int_{s}^{t_{k-1}} g''(t_m-\theta)(t_{k-1}-\theta)d\theta   \Biggl] \p_{s}^2 u(x_i,y_j,s)ds.
\end{align*}
For $(R_{11})_i^m$, based on Lemma \ref{gpro} and the regularity assumptions, we yield
\begin{equation*}
\begin{aligned}
\left\|(R_{11})^m\right\|_{h,\infty}
\le & Q_g\sum_{k=1}^{m-1}\int_{t_{k-1}}^{t_{k}}\frac{t_{k-1}-s}{\tau}\int_{s}^{t_{k}}(t_{m}-\theta)^{-\varepsilon}(t_{k}-\theta)d\theta \left\|\p_{s}^2 u(\cdot,s)\right\|_{h,\infty}ds\\
&+ Q_g\int_{t_{m-1}}^{t_{m}}\frac{t_{m-1}-s}{\tau}\int_{s}^{t_{m}}(t_{m}-\theta)^{1-\varepsilon}d\theta \left\|\p_{s}^2 u(\cdot,s)\right\|_{h,\infty}ds\\
\le & Q\tau^{2}\sum_{k=1}^{m-1}\int_{t_{k-1}}^{t_{k}}(t_{m}-s)^{-\varepsilon}s^{\alpha_{0}-2}ds
+ Q\tau\int_{t_{m-1}}^{t_{m}}(t_{m}-s)^{1-\varepsilon}s^{\alpha_{0}-2}ds\\
\le & Q\tau^{2}\int_{0}^{t_{m}}(t_{m}-s)^{-\varepsilon}s^{\alpha_{0}-2}ds 
+ Q\tau^{2-\varepsilon}\int_{t_{m-1}}^{t_{m}}s^{\alpha_{0}-2}ds\\
\le &  Q\left[\tau^{2}t_{m}^{\alpha_{0}-(1+\varepsilon)}
+\tau^{2-\varepsilon}\left(t_{m}^{\alpha_{0}-1}-t_{m-1}^{\alpha_{0}-1}\right)\right].
\end{aligned}
\end{equation*}
Applying a similar argument to $(R_{12})_i^m$ leads to the same estimate, and hence we obtain
\begin{equation*}
    \tau \sum\limits_{m=1}^{n} \left\| (R_1)^m \right\|_{h,\infty}
\leq Q\left(\tau^{2}+\tau^{3-\varepsilon}T^{\alpha_0-1}\right)\le Q\tau^2,
\end{equation*}
which, together with Lemma \ref{space1} and triangle inequality, implies
\begin{equation}\label{3.13-1}
  \tau\sum_{m=1}^n\|\mathcal{A}_{h}(R_{1})^{m-\frac{1}{2}}\|_{h,\infty}\le \tau\sum_{m=1}^n\|(R_{1})^{m-\frac{1}{2}}\|_{h,\infty}\le Q\tau^{2}.
\end{equation}
Then, from \eqref{R2}, Taylor expansion gives
\begin{equation*}
\begin{aligned}
(R_2)_{i,j}^n=&\frac{1}{\tau}\int_{t_{n-1}}^{t_n}\widehat{G}(x_{i},y_j,t)dt-\widehat{G}(x_{i},y_j,t_{n-\frac{1}{2}})+\widehat{G}(x_{i},y_j,t_{n-\frac{1}{2}})-\frac{\widehat{G}(x_{i},y_j,t_n)+\widehat{G}(x_{i},y_j,t_{n-1})}{2}\\
=&\frac{1}{2\tau}\int_{t_{n-1}}^{t_{n-\frac{1}{2}}}(t-t_{n-1})^{2}\partial_{t}^{2}\widehat{G}(x_{i},y_j,t)dt+\frac{1}{2\tau}\int_{t_{n-\frac{1}{2}}}^{t_{n}}(t_{n}-t)^{2}\partial_{t}^{2}\widehat{G}(x_{i},y_j,t)dt\\
&+\frac{1}{2}\int_{t_{n-1}}^{t_{n-\frac{1}{2}}}(t_{n-1}-t)\partial_{t}^{2}\widehat G(x_{i},y_j,t)dt+\frac{1}{2}\int_{t_{n-\frac{1}{2}}}^{t_{n}}(t-t_{n})\partial_{t}^{2}\widehat G(x_{i},y_j,t)dt,
\end{aligned}
\end{equation*}
which, together with the estimate in \eqref{Gpro}, leads to
\begin{equation*}
\begin{aligned}
\tau\sum_{m=1}^n\|(R_{2})^{m}\|_{h,\infty}\leq&\tau\int_{0}^{t_{\frac{1}{2}}}t\left\|\p_{t}^{2}\widehat{G}(\cdot,t)\right\|_{h,\infty}dt+\tau^{2}\int_{t_{\frac{1}{2}}}^{t_{1}}\left\|\p_{t}^{2}\widehat{G}(\cdot,t)\right\|_{h,\infty}dt\\
&\quad +2\sum_{m=2}^{n}\tau^{2}\int_{t_{m-1}}^{t_{m}}\left\|\p_{t}^{2}\widehat{G}(\cdot,t)\right\|_{h,\infty}dt \\
\leq&Q\tau\int_{0}^{t_{\frac{1}{2}}}t^{\alpha_{0}-1}dt+Q\tau^{2}\int_{t_{\frac{1}{2}}}^{t_{1}}t^{\alpha_{0}-2}dt+Q\tau^{2}\sum_{m=2}^{n}\int_{t_{m-1}}^{t_{m}}t^{\alpha_{0}-2}dt\\
\leq&Q\tau^{2},
\end{aligned}
\end{equation*}
thus we have
\begin{equation}\label{3.13-2}
\tau\sum_{m=1}^n\|\mathcal{A}_{h}(R_{2})^{m}\|_{h,\infty}\le \tau\sum_{m=1}^n\|(R_{2})^{m}\|_{h,\infty}\le Q\tau^{2}.
\end{equation}

Next, we focus on the bounds of $(R_{3})^{n}$ and $(R_{4})^{n}$. Split 
\eqref{R3} as follows
\begin{equation*}
\begin{aligned}
(R_{3})_{i,j}^{n}=&\frac{1}{\tau}\int_{t_{n-1}}^{t_{n}}\int_{0}^{t}\beta_{\bar{\alpha}}(t-s)\left[\Delta u(x_{i},y_{j},s)-\Delta u^{*}(x_{i},y_{j},s)\right]dsdt\\
&+\frac{1}{\tau}\int_{t_{n-1}}^{t_{n}}\int_{0}^{t}\beta_{\bar{\alpha}}(t-s)\left[\Delta u^{*}(x_{i},y_{j},s)-\Delta\breve{u}(x_{i},y_{j},s)\right]dsdt,
\end{aligned}
\end{equation*}
with
\begin{equation*}
u^*(x_{i},y_{j},s)=
\begin{cases}
u_{i,j}^{1},&0<s<\tau,\\
\frac{1}{\tau}[(t_i-s)u_{i,j}^{n-1}+(s-t_{i-1})u_{i,j}^{n}],&t_{n-1}<s<t_n,\quad n\ge2.
\end{cases}
\end{equation*}
Subsequently, from \cite{Mclean}, we yield
\begin{equation}\nonumber
\begin{aligned}
\tau\sum_{m=1}^n\|(R_3)^m\|_{h,\infty}\le& Q_{\bar{\alpha},T}\left(\int_{0}^{t_{1}}t\left\|\Delta\partial_{t}u(\cdot,t)\right\|_{h,\infty}dt+\tau^{2}\int_{t_{1}}^{t_{n}}\left\|\Delta\partial_{t}^{2}u(\cdot,t)\right\|_{h,\infty}dt\right)\\
&+Q_{\bar{\alpha},T}\left(\tau\int_{t_{1}}^{t_{2}}\|\Delta\partial_{t}u(\cdot,t)\|_{h,\infty}dt+\tau^{2}\int_{t_{1}}^{t_{n}}\left\|\Delta\partial_{t}^{2}u(\cdot,t)\right\|_{h,\infty}dt\right)\\
\leq &Q\tau\int_0^{2\tau}t^{\alpha_0-1}dt+Q\tau^2\int_\tau^{t_n}t^{\alpha_0-2}dt
\leq Q\tau^2,
\end{aligned}
\end{equation}
thus, 
\begin{equation}\label{3.13-3}
\tau\sum_{m=1}^n\|\mathcal{A}_{h}(R_{3})^{m}\|_{h,\infty}\le \tau\sum_{m=1}^n\|(R_{3})^{m}\|_{h,\infty}\le Q\tau^{2}.
\end{equation}
Similarly, by combining this with the chain rule, we obtain 
\begin{equation}\label{3.13-4}
\begin{aligned}
\tau\sum_{m=1}^n\|\mathcal{A}_{h}(R_4)^m\|_{h,\infty} & \leq Q_{\bar{\alpha},T}Q_{f}\left(\int_{0}^{t_{1}}t\left\|\partial_{t}u(\cdot,t)\right\|_{h,\infty}dt+\tau^{2}\int_{t_{1}}^{t_{n}}\left\|\partial_{t}^{2}u(\cdot,t)\right\|_{h,\infty}dt\right)\\
&+Q_{\bar{\alpha},T}Q_{f}\left(\tau\int_{t_{1}}^{t_{2}}\|\partial_{t}u(\cdot,t)\|_{h,\infty}dt+\tau^{2}\int_{t_{1}}^{t_{n}}\left\|\partial_{t}^{2}u(\cdot,t)\right\|_{h,\infty}dt\right)
\leq  Q\tau^2.
\end{aligned}
\end{equation}
At last, with the help of \eqref{lam1}-\eqref{lam3}, we have 
\begin{equation}\label{3.13-5}
\begin{aligned}
&\tau\sum_{m=1}^{n}\mathfrak{I}_{\bar{\alpha}}^{m}\|(R_{x})^{m}\|_{h,\infty} \\
=&\tau\left[\lambda_{1,1}\|(R_{x})^{1}\|_{h,\infty}+\sum_{m=2}^{n}\left(\lambda_{m,1}\|(R_{x})^{1}\|_{h,\infty}+\sum_{k=2}^{m}\lambda_{m,k}\|(R_{x})^{k-\frac{1}{2}}\|_{h,\infty}\right)\right] \\
\leq& Q\tau h_{x}^4\sum_{m=1}^n\sum_{k=1}^m\lambda_{m,k}\leq Qh_{x}^4,
\end{aligned}  
\end{equation}
and 
\begin{equation}\label{3.13-6}
\tau\sum_{m=1}^{n}\mathfrak{I}_{\bar{\alpha}}^{m}\|(R_{y})^{m}\|_{h,\infty}\leq Qh_{y}^4.
\end{equation}

Taking \eqref{3.13-1}, \eqref{3.13-2}, \eqref{3.13-3}, \eqref{3.13-4}, \eqref{3.13-5} and \eqref{3.13-6} into \eqref{e3.13}, it turns into
\begin{equation*}
  \tau\sum_{m=1}^n\|(R)^m\|_{h,\infty}\leq Q(\tau^2+h_x^4+h_y^4).
\end{equation*}
Since the truncation error on the coarse grid can be estimated in a similar manner, we omit the details here. This completes the proof.
\end{proof}

Next, omitting the local truncation error term $R_{i,j}^{n}$ in
\eqref{e3.12}, and replacing $u_{i,j}^{n}$ with its numerical approximation
$U_{i,j}^{n}$, we establish a standard compact difference scheme on the fine grid 
\begin{equation}\label{standard fine}
\begin{aligned}
&[\mathcal{A}_{h}\delta_{t}U]_{i,j}^{n}+\sum_{k=1}^{n}\tilde{w}_{n-k}[\mathcal{A}_{h}\delta_{t}U]_{i,j}^{k}-\mathfrak{I}_{\bar{\alpha}}^{n}[\Lambda_{h} U]_{i,j}^{n}=\mathfrak{I}_{\bar{\alpha}}^{n}[\mathcal{A}_{h}f(U_{i,j}^{n})]+[\mathcal{A}_{h}b]_{i,j}^{n},\quad (i,j)\in\omega_{h},~~1\le n\le N,\\
&U_{i,j}^0=u_{0}(x_{i},y_{j}),\qquad(i,j)\in\omega_{h},\\
&U_{i,j}^n=0,\qquad(i,j)\in\partial\omega_{h},\quad 0\leq n\leq N, 
\end{aligned}
\end{equation}
and a standard compact difference scheme on the coarse grid 
\begin{equation}\label{standard coarse}
\begin{aligned}
&[\mathcal{A}_{H}\delta_{t}U]_{i,j}^{n}+\sum_{k=1}^{n}\tilde{w}_{n-k}[\mathcal{A}_{H}\delta_{t}U]_{i,j}^{k}-\mathfrak{I}_{\bar{\alpha}}^{n}[\Lambda_{H} U]_{i,j}^{n}=\mathfrak{I}_{\bar{\alpha}}^{n}[\mathcal{A}_{H}f(U_{i,j}^{n})]+[\mathcal{A}_{H}b]_{i,j}^{n},\quad (i,j)\in\omega_{H},~~1\le n\le N,\\
&U_{i,j}^0=u_{0}(x_{i},y_{j}),\qquad(i,j)\in\omega_{H},\\
&U_{i,j}^n=0,\qquad(i,j)\in\partial\omega_{H},\quad 0\leq n\leq N.
\end{aligned}
\end{equation}

\subsection{Construction of STG compact difference scheme}
The key ingredient in constructing the STG scheme is the high-order mapping operator that transfers solutions from the coarse-grid space to the fine-grid space. Thus, we introduce below the 2D bicubic spline interpolation operator and summarize its main properties, as discussed in \cite{fuhongfei}.

To this end, we first review the one-dimensional cubic spline interpolation operator
along $x$-direction. For any continuous function $w(x)$, with the notation 
$w_{i}=w(x_{i})$ for $0\le i\le M_{H,x}$, the cubic spline 
interpolation operator is defined as 
\begin{equation*}
\begin{aligned}
\Pi_{H,x}w(x) =&\mathcal{M}_{i-1}\frac{(x_{i}-x)^{3}}{6H_{x}}+\mathcal{M}_{i}\frac{(x-x_{i-1})^{3}}{6H_{x}}+\left(w_{i-1}-\frac{\mathcal{M}_{i-1}H_{x}^{2}}{6}\right)\frac{x_{i}-x}{H_{x}} \\
&+\left(w_{i}-\frac{\mathcal{M}_{i}H_{x}^{2}}{6}\right)\frac{x-x_{i-1}}{H_{x}},\quad x\in[x_{i-1},x_{i}],\quad i=1,2,\cdots,M_{H,x}.
\end{aligned}
\end{equation*}
Here, the coefficients 
$\mathbf{M}=[\mathcal{M}_1,\mathcal{M}_2,\cdots,\mathcal{M}_{M_{H,x}-1}]^\top$ are 
given by the linear algebraic systems $\mathbf{AM=D}$ where 
$\mathbf{A}=\frac{1}{2}\text{tridiag}(1,4,1)$ and 
$\mathbf{D}=[D_1,D_2,\cdots,D_{M_{H,x}-1}]^\top$ with 
$D_1=3[\delta_{H,x}^{2}w]_1-\frac{1}{2}\mathcal{M}_0$, 
$D_{M_{H,x}-1}=3[\delta_{H,x}^{2}w]_{M_{H,x}-1}-\frac{1}{2}\mathcal{M}_{M_{H,x}}$ and 
$D_i=3[\delta_{H,x}^{2}w]_i$ for $i=2,3,\cdots,M_{H,x}-2$. The coefficients $\mathcal{M}_0$ 
and $\mathcal{M}_{M_{H,x}}$ represent the boundary values of the second-order derivatives of the 
interpolation function and need to be specified based on a priori information, such 
as the second-order boundary condition of \eqref{VtFDEs}-\eqref{bc}.

Similarly, we can define the one-dimensional cubic spline interpolation operator 
$\Pi_{H,y}$ along $y$-direction. Thus, the 2D bicubic spline 
interpolation operator could be defined as the superposition of one-dimensional 
cubic spline interpolation operators in different directions, i.e., 
$\Pi_{H}=\Pi_{H,y}\Pi_{H,x}$. When no confusion caused, in the rest of this work, 
we denote $(\Pi_Hw)_{i,j}=\Pi_Hw_{i,j}$ for $(i,j)\in\bar{\omega}_h$ and $w\in\mathfrak{U}_{H}$.

Next, we present several important properties of operator $\Pi_H$.

\begin{lemma}\label{inteesti1}
For any function $v\in W^{4,\infty}(\Omega)$, there exists a positive constant $Q$ 
independent of $H_{x}$ and $H_{y}$ such that
\begin{equation*}
  \|v-\Pi_H v\|_{h,\infty} \leq Q\left(H_x^4+H_y^4\right).
\end{equation*}
\end{lemma}

\begin{lemma}\label{inteesti2}
For any grid function $v\in \dot{\mathfrak{U}}_{H}$, if the boundary values of the interpolation 
coefficients of $\Pi_H v$ are zero, we have $\|\Pi_{H}v\|_h\leq 48\|v\|_H$.
\end{lemma}

At this point, on the basis of standard compact difference schemes \eqref{standard fine} 
and \eqref{standard coarse}, we apply the high-order mapping operator
$\Pi_H$ to establish STG compact difference scheme for \eqref{VtFDEs}-\eqref{bc} 
as follows:\par 
\textbf{Step 1.} On the coarse grid, we compute the following small-scale nonlinear system to 
find the rough solution $U_{H}^{n}$ 
\begin{align}
&c_{0}[\mathcal{A}_{H}\delta_{t}U_{H}]_{i,j}^{1}-\lambda_{1,1}[\Lambda_{H} U_{H}]_{i,j}^{1}=\lambda_{1,1}\mathcal{A}_{H}f([U_{H}]_{i,j}^{1})+[\mathcal{A}_{H}b]_{i,j}^{1},\quad (i,j)\in\omega_{H},\label{coarse-scheme-1}\\
&c_{0}[\mathcal{A}_{H}\delta_{t}U_{H}]_{i,j}^{n}+\sum_{k=1}^{n-1}\tilde{w}_{n-k}[\mathcal{A}_{H}\delta_{t}U_{H}]_{i,j}^{k}-\mathfrak{I}_{\bar{\alpha}}^{n}[\Lambda_{H} U_H]_{i,j}^{n}=\frac{\lambda_{n,n}\mathcal{A}_{H}f([U_{H}]_{i,j}^{n})}{2}\nonumber\\
&\qquad\quad +\sum_{k=1}^{n-1}\widetilde{\lambda}_{n,k}\mathcal{A}_{H}f([U_{H}]_{i,j}^{k})+[\mathcal{A}_{H}b]_{i,j}^{n},\qquad (i,j)\in\omega_{H},\quad 2\le n\le N,\label{coarse-scheme-2}\\
&[U_{H}]_{i,j}^0=u_{0}(x_{i},y_{j}),\quad [\p_{t}U_{H}]_{i,j}^0=\bar{u}_{0}(x_{i},y_{j}),\qquad(i,j)\in\omega_{H},\label{coarse-scheme-3}\\
&[U_{H}]_{i,j}^n=0,\qquad(i,j)\in\partial\omega_{H},\quad 0\leq n\leq N,\label{coarse-scheme-4}
\end{align}
with $c_0 = 1 + \frac{w_0}{2}$ and 
\begin{equation*}
\widetilde{\lambda}_{n,1}=\lambda_{n,1}+\frac{\lambda_{n,2}}{2},\qquad \widetilde{\lambda}_{n,k}=\frac{1}{2}\left(\lambda_{n,k+1}+\lambda_{n,k}\right),\quad 2\le k\le n-1.
\end{equation*}

\textbf{Step 2.} On the fine grid, with the help of the rough solution $U_{H}^{n}$ 
and bicubic spline interpolation operator $\Pi_{H}$, we turn to solve a large-scale 
linear system to obtain the accurate numerical solution $U_{h}^{n}$ 
\begin{align}
&c_{0}[\mathcal{A}_{h}\delta_{t}U_{h}]_{i,j}^{1}-\lambda_{1,1}[\Lambda_{h} U_{h}]_{i,j}^{1}=\lambda_{1,1}[\mathcal{A}_{h}F]_{i,j}^{1}+[\mathcal{A}_{h}b]_{i,j}^{1},\quad (i,j)\in\omega_{h},\label{fine-scheme-1}\\
&c_{0}[\mathcal{A}_{h}\delta_{t}U_{h}]_{i,j}^{n}+\sum_{k=1}^{n-1}\tilde{w}_{n-k}[\mathcal{A}_{h}\delta_{t}U_{h}]_{i,j}^{k}-\mathfrak{I}_{\bar{\alpha}}^{n}[\Lambda_{h} U_h]_{i,j}^{n}=\frac{\lambda_{n,n}}{2}[\mathcal{A}_{h}F]_{i,j}^{n}\nonumber\\
&\qquad +\sum_{k=1}^{n-1}\widetilde{\lambda}_{n,k}\mathcal{A}_{h}f([U_{h}]_{i,j}^{k})+[\mathcal{A}_{h}b]_{i,j}^{n},\qquad (i,j)\in\omega_{h},\quad 2\le n\le N,\label{fine-scheme-2}\\
&[U_{h}]_{i,j}^0=u_{0}(x_{i},y_{j}),\quad [\p_{t}U_{h}]_{i,j}^0=\bar{u}_{0}(x_{i},y_{j}),\qquad(i,j)\in\omega_{h},\label{fine-scheme-3}\\
&[U_{h}]_{i,j}^n=0,\qquad(i,j)\in\partial\omega_{h},\quad 0\leq n\leq N,\label{fine-scheme-4}
\end{align}
where $F_{i,j}^n$ represents the Newton linearization of $f$ via 
\begin{equation}\nonumber
  F_{i,j}^n = f([\Pi_H U_H]_{i,j}^n) + f'([\Pi_H U_H]_{i,j}^n)([U_h]_{i,j}^n-[\Pi_H U_H]_{i,j}^n).
\end{equation}

\section{Analysis of STG compact difference scheme}\label{sec4}
In this section, our objective is to discuss the stability and convergence of the 
STG compact difference scheme \eqref{coarse-scheme-1}-\eqref{fine-scheme-4}.

\subsection{Stability}
We first use the energy argument to consider the stability of the scheme  \eqref{coarse-scheme-1}- \eqref{coarse-scheme-4} on the coarse grid.

\begin{theorem}\label{staboncoarse}
If the temporal stepsize $\tau$ is sufficiently small, the numerical solution 
of nonlinear compact difference system \eqref{coarse-scheme-1}-\eqref{coarse-scheme-4} 
on the coarse grid is stable, and there exists a positive constant 
$Q=Q(T,\bar{\alpha},Q_g,Q_f)$ such that 
\begin{equation}\nonumber
\left\|U_{H}^{n}\right\|_{H}\le Q\left(\left\|u_0\right\|_{H}+\left\|f(u_0)\right\|_H+\left\|\bar{u}_0\right\|_H\right), \quad 1\leq n \leq N.
\end{equation}
\end{theorem}
\begin{proof}
When $n=1$, taking the inner product $(\cdot , \cdot)_{H}$ of equation 
\eqref{coarse-scheme-1} with $2\tau U_H^1$, we have
\begin{equation}\label{e4.1}
2c_{0}\tau\left(\mathcal{A}_{H}\delta_{t}U_{H}^{1},U_{H}^{1}\right)_{H}-2\lambda_{1,1}\tau\left(\Lambda_{H} U_{H}^{1},U_{H}^{1}\right)_{H}
=2\lambda_{1,1}\tau\left(\mathcal{A}_{H}f(U_{H}^{1}),U_{H}^{1}\right)_{H}+2\tau\left(\mathcal{A}_{H}b^{1},U_{H}^{1}\right)_{H}.
\end{equation}
We first get 
\begin{equation*}
2\tau\left(\mathcal{A}_{H}\delta_t U_H^1,U_H^1\right)_{H} \geq \|U_H^1\|^2_{\mathcal{A}_H}-\|U_H^0\|^2_{\mathcal{A}_H},
\end{equation*}
and 
\begin{equation*}
\begin{aligned}
2\lambda_{1,1}\tau\left(\mathcal{A}_{H}f(U_H^1),U_H^1\right)_{H}&\leq 2\tau Q_{\bar{\alpha},T}\|\mathcal{A}_{H}f(U_H^1)\|_H \|U_H^1\|_H \\
&\leq 2\tau Q_{\bar{\alpha},T}\left(\|f(U_H^1)-f(U_H^0)\|_H+\|f(U_H^0)\|_H \right)\|U_H^1\|_H\\
&\leq 2\tau Q_{\bar{\alpha},T}\left(Q_{f}\|U_H^1-U_H^0\|_H +\|f(U_H^0)\|_H \right)\|U_H^1\|_H,
\end{aligned}  
\end{equation*}
where we used the estimate of $\lambda_{1,1}$ in \eqref{lam2}. Then, \eqref{e4.1} 
turns into 
\begin{equation}\label{e4.2}
\begin{aligned}
&c_{0}\|U_H^1\|^2_{\mathcal{A}_H}-2\lambda_{1,1}\tau\left(\Lambda_{H} U_{H}^{1},U_{H}^{1}\right)_{H}\\
\le& c_{0}\|U_H^0\|^2_{\mathcal{A}_H}+2\tau\left[\|\mathcal{A}_Hb^1\|_H+Q_{\bar{\alpha},T}\left(Q_{f}\|U_H^1-U_H^0\|_H +\|f(U_H^0)\|_H \right)\right]\|U_H^1\|_H.
\end{aligned}  
\end{equation}
Next, taking the inner product $(\cdot , \cdot)_{H}$ of equation 
\eqref{coarse-scheme-2} with $2\tau U_H^{n-\frac{1}{2}}$ and summing up 
for $n$ from $2$ to $m$ for $2\le m\le N$, we get 
\begin{equation}\label{e4.3}
\begin{aligned}
&2c_{0}\tau\sum_{n=2}^{m}\left(\mathcal{A}_{H}\delta_{t}U_{H}^{n},U_{H}^{n-\frac{1}{2}}\right)_{H}-2\tau\sum_{n=2}^{m}\left(\mathfrak{I}_{\bar{\alpha}}^{n}\Lambda_{H} U_H^{n},U_{H}^{n-\frac{1}{2}}\right)_{H}\\
=&-2\tau\sum_{n=2}^{m}\sum_{k=1}^{n-1}\tilde{w}_{n-k}\left(\mathcal{A}_{H}\delta_{t}U_{H}^{k},U_{H}^{n-\frac{1}{2}}\right)_{H}+\tau\sum_{n=2}^{m}\lambda_{n,n}\left(\mathcal{A}_{H}f(U_{H}^{n}),U_{H}^{n-\frac{1}{2}}\right)_{H}\\
&+2\tau\sum_{n=2}^{m}\sum_{k=1}^{n-1}\widetilde{\lambda}_{n,k}\left(\mathcal{A}_{H}f(U_{H}^{k}),U_{H}^{n-\frac{1}{2}}\right)_{H}+2\tau\sum_{n=2}^{m}\left(\mathcal{A}_{H}b^{n},U_{H}^{n-\frac{1}{2}}\right)_{H}.
\end{aligned}
\end{equation}
In the following, we shall estimate each term in \eqref{e4.3}. To begin with, it is straightforward to observe that 
\begin{equation}\label{e4.4}
2\tau\sum_{n=2}^{m}\left(\mathcal{A}_{H}\delta_{t}U_{H}^{n},U_{H}^{n-\frac{1}{2}}\right)_{H} =\|U_{H}^{m}\|_{\mathcal A_H}^{2}-\|U_{H}^{1}\|_{\mathcal A_H}^{2}.   
\end{equation}
Then, for the second term on the left-hand side, taking the fact
\begin{equation*}
\begin{aligned}
\sum_{k=1}^{n-1}\tilde{w}_{n-k}\delta_{t}U_{H}^{k}&=\frac{1}{\tau}\left[\tilde{w}_{1}U_{H}^{n-1}+\sum_{k=1}^{n-2}(\tilde{w}_{n-k}-\tilde{w}_{n-k-1})U_{H}^{k}-\tilde{w}_{n-1}U_{H}^{0}\right] \\
&=\frac{1}{\tau}\left[\tilde{w}_{1}U_{H}^{n-1}+\sum_{k=2}^{n-1}(\tilde{w}_{k}-\tilde{w}_{k-1})U_{H}^{n-k}-\tilde{w}_{n-1}U_{H}^{0}\right] 
\end{aligned}  
\end{equation*}
into consideration, with the help of Cauchy-Schwarz inequality, we obtain
\begin{equation}\label{e4.5}
\begin{aligned}
&-2\tau\sum_{n=2}^m\sum_{k=1}^{n-1}\tilde{w}_{n-k}\left(\mathcal{A}_H\delta_t U_{H}^{k},U_{H}^{n-\frac{1}{2}}\right)_{H} \\
\leq&\sum_{n=2}^{m}|\tilde{w}_{1}|\left\|U_H^{n-1}\right\|_{\mathcal{A}_{H}}\left\|U_H^{n}\right\|_{\mathcal{A}_{H}}+\sum_{n=3}^{m}\sum_{k=2}^{n-1}|\tilde{w}_{k}-\tilde{w}_{k-1}|\left\|U_{H}^{n-k}\right\|_{\mathcal{A}_{H}}\left\|U_H^{n}\right\|_{\mathcal{A}_{H}} \\
&+\sum_{n=2}^{m}|\tilde{w}_{1}|\left\|U_H^{n-1}\right\|_{\mathcal{A}_{H}}^{2}+\sum_{n=3}^{m}\sum_{k=2}^{n-1}|\tilde{w}_{k}-\tilde{w}_{k-1}|\left\|U_{H}^{n-k}\right\|_{\mathcal{A}_{H}}\left\|U_H^{n-1}\right\|_{\mathcal{A}_{H}}\\
&+\sum_{n=2}^{m}|\tilde{w}_{n-1}|\left\|U_H^{0}\right\|_{\mathcal{A}_{H}}\left\|U_H^{n}\right\|_{\mathcal{A}_{H}}+\sum_{n=2}^{m}|\tilde{w}_{n-1}|\left\|U_H^{0}\right\|_{\mathcal{A}_{H}}\left\|U_H^{n-1}\right\|_{\mathcal{A}_{H}}.
\end{aligned}
\end{equation}
Meanwhile, the application of triangle inequality leads to 
\begin{equation}\label{e4.6}
\begin{aligned}
&\tau\sum_{n=2}^{m}\lambda_{n,n}\left(\mathcal{A}_{H}f(U_{H}^{n}),U_{H}^{n-\frac{1}{2}}\right)_{H}+2\tau\sum_{n=2}^{m}\sum_{k=1}^{n-1}\widetilde{\lambda}_{n,k}\left(\mathcal{A}_{H}f(U_{H}^{k}),U_{H}^{n-\frac{1}{2}}\right)_{H}\\
\le &\tau Q_{\bar{\alpha},T}\sum_{n=2}^{m}\left\|\mathcal{A}_Hf(U_H^n)\right\|_H\left\|U_H^{n-\frac{1}{2}}\right\|_H+2\tau\sum_{n=2}^{m}\sum_{k=1}^{n-1}\widetilde{\lambda}_{n,k}\left\|\mathcal{A}_{H}f(U_{H}^{k})\right\|_{H}\left\|U_{H}^{n-\frac{1}{2}}\right\|_{H}\\
\leq&\tau Q_{\bar{\alpha},T}\sum_{n=2}^m\left(Q_f\left\|U_H^n-U_H^0\right\|_H+\left\|f(U_H^0)\right\|_H\right)\left\|U_H^{n-\frac12}\right\|_H\\
&+\tau \sum_{n=2}^m\sum_{k=1}^{n-1}\left[\widetilde{\lambda}_{n,k}\left(Q_f\left\|U_H^{k}-U_H^0\right\|_H+\left\|f(U_H^0)\right\|_H\right)\right]\left(\left\|U_H^{n}\right\|_H+\left\|U_H^{n-1}\right\|_H\right).
\end{aligned}  
\end{equation}
Taking \eqref{e4.4}, \eqref{e4.5} and \eqref{e4.6} into \eqref{e4.3}, in combination with 
\eqref{e4.2}, we have 
\begin{equation}\label{e4.7}
\begin{aligned}
&c_{0}\|U_{H}^{m}\|_{\mathcal A_H}^{2}-2\lambda_{1,1}\tau\left(\Lambda_{H} U_{H}^{1},U_{H}^{1}\right)_{H}-2\tau\sum_{n=2}^{m}\left(\mathfrak{I}_{\bar{\alpha}}^{n}\Lambda_{H} U_H^{n},U_{H}^{n-\frac{1}{2}}\right)_{H}\\
\le &c_{0}\left\|U_H^0\right\|^2_{\mathcal{A}_H}+\sum_{n=3}^{m}\sum_{k=2}^{n-1}|\tilde{w}_{k}-\tilde{w}_{k-1}|\left\|U_{H}^{n-k}\right\|_{H}\left(\left\|U_H^{n}\right\|_{H}+\left\|U_H^{n-1}\right\|_{H}\right) \\
&+\sum_{n=2}^{m}|\tilde{w}_{1}|\left\|U_H^{n-1}\right\|_{H}\left(\left\|U_H^{n}\right\|_{H}+\left\|U_H^{n-1}\right\|_H\right)+2\tau\left\|\mathcal{A}_Hb^1\right\|_H\left\|U_H^1\right\|_H\\
&+\sum_{n=2}^{m}|\tilde{w}_{n-1}|\left\|U_H^{0}\right\|_{H}\left(\left\|U_H^{n}\right\|_{H}+\left\|U_H^{n-1}\right\|_{H}\right)+2\tau\sum_{n=2}^{m}\left\|\mathcal{A}_Hb^n\right\|_H\left\|U_{H}^{n-\frac{1}{2}}\right\|_{H}\\
&+2\tau Q_{\bar{\alpha},T}\left(Q_{f}\|U_H^1-U_H^0\|_H +\|f(U_H^0)\|_H \right)\|U_H^1\|_H\\
&+\tau Q_{\bar{\alpha},T}\sum_{n=2}^m\left(Q_f\left\|U_H^n-U_H^0\right\|_H+\left\|f(U_H^0)\right\|_H\right)\left\|U_H^{n-\frac12}\right\|_H\\
&+\tau \sum_{n=2}^m\sum_{k=1}^{n-1}\left[\widetilde{\lambda}_{n,k}\left(Q_f\left\|U_H^{k}\right\|+Q_{f}\left\|U_H^0\right\|_H+\left\|f(U_H^0)\right\|_H\right)\right]\left(\left\|U_H^{n}\right\|_H+\left\|U_H^{n-1}\right\|_H\right).
\end{aligned}  
\end{equation}

Next, we focus on simplifying this inequality. First, by applying Lemmas \ref{positivityofLambda} and \ref{PIpositive}, we obtain the fact that
\begin{equation}\label{positive}
 \begin{aligned}
&-\left(\mathfrak{I}_{\bar{\alpha}}^{1}\Lambda_{H}U_H^1,U_H^1\right)_H-\sum_{n=2}^{m}\left(\mathfrak{I}_{\bar{\alpha}}^{n}\Lambda_{H}U_H^n,U_H^{n-\frac{1}{2}}\right)_H \\
=&\left(\mathfrak{I}_{\bar{\alpha}}^{1}\nabla_H^*U_H^1,\nabla_H^*U_H^1\right)_H
+\sum_{n=2}^{m}\left(\mathfrak{I}_{\bar{\alpha}}^{n}\nabla_H^*U_{H}^{n},\nabla_H^*U_H^{n-\frac{1}{2}}\right)_H \ge 0.
\end{aligned} 
\end{equation}
Then, choosing a suitable $\mathcal{K}$ such that 
$\|U_{H}^{\mathcal{K}}\|_{H}=\max\limits_{1\leq n\leq N}\|U_{H}^{n}\|_{H}$, \eqref{e4.7}
can be transformed into
\begin{equation}\label{e4.9}
\begin{aligned}
\frac{c_{0}}{3}\|U_{H}^{\mathcal{K}}\|_{H}\le &c_{0}\left\|U_H^0\right\|_{H}+2\tau Q_{\bar{\alpha},T}\sum_{n=1}^{\mathcal{K}}\left(Q_f\left\|U_H^n-U_H^0\right\|_H+\left\|f(U_H^0)\right\|_H\right)\\
&+2\sum_{n=2}^{\mathcal{K}}|\tilde{w}_{1}|\left\|U_H^{n-1}\right\|_{H}+\sum_{n=3}^{\mathcal{K}}\sum_{k=2}^{n-1}|\tilde{w}_{k}-\tilde{w}_{k-1}|\left(\left\|U_{H}^{n-1}\right\|_{H}+\left\|U_H^{n}\right\|_{H}\right) \\
&+2\sum_{n=2}^{\mathcal{K}}|\tilde{w}_{n-1}|\left\|U_H^{0}\right\|_{H}+\tau Q_f\sum_{n=2}^{\mathcal{K}}\sum_{k=1}^{n-1}\widetilde{\lambda}_{n,k}\left(\left\|U_H^{n-1}\right\|_H+\left\|U_H^{n}\right\|_H\right)\\
&+2\tau\sum_{n=2}^{\mathcal{K}}\sum_{k=1}^{n-1}\widetilde{\lambda}_{n,k}\left(Q_f\left\|U_H^{0}\right\|_H+\left\|f(U_H^0)\right\|_H\right)+2\tau\sum_{n=1}^{\mathcal{K}}\left\|b^n\right\|_H.
\end{aligned}  
\end{equation}
If we want to continue simplifying \eqref{e4.9}, we must obtain  the estimates of 
coefficients in this formula. Here, Lemma \ref{gpro} and Taylor expansion formula 
lead to 
\begin{equation}\label{wpro1}
\begin{aligned}
\sum_{k=2}^{n-1}|\tilde{w}_{k}-\tilde{w}_{k-1}|& =\frac12\sum_{k=2}^{n-1}|w_{k}-w_{k-1}+w_{k-1}-w_{k-2}|\leq\sum_{k=1}^{n-1}|w_{k}-w_{k-1}|  \\
&\le \sum_{k=1}^{n-1}\left(\int_{t_k}^{t_{k+1}}(t_{k+1}-t)|g^{\prime\prime}(t)|dt+\int_{t_{k-1}}^{t_k}(t-t_{k-1})|g^{\prime\prime}(t)|dt\right)\\
&\le \tau\sum_{k=1}^{n-1}\int_{t_{k-1}}^{t_{k+1}}|g^{\prime\prime}(t)|dt
\le Q_g\tau\int_{0}^{t_n}t^{-\varepsilon}\le \tau Q_{g}Q_{T}.
\end{aligned}  
\end{equation}
Meanwhile, it is obvious that 
\begin{align}
&|w_0|\leq\int_0^\tau|g'(t)|dt\leq \tau Q_g,\label{wpro2}\\
&|\tilde{w}_1|=|\frac{w_1+w_0}{2}|\leq\frac{1}{2}\int_0^{2\tau}|g'(t)|dt\leq \tau Q_g,\label{wpro3}\\
&|\tilde{w}_{n-1}|=|\frac{w_{n-1}+w_{n-2}}{2}|\leq\frac{1}{2}\int_{t_{n-2}}^{t_{n}}|g'(t)|dt\leq \tau Q_g.\label{wpro4}
\end{align}
Based on \eqref{lam3}, \eqref{wpro1}, \eqref{wpro2}, \eqref{wpro3} and \eqref{wpro4}, we 
further get 
\begin{equation}\nonumber
\begin{aligned}
\left(1-\frac{Q_g}{2}\tau\right)\|U_{H}^{N}\|_{H}\le &\left(3+\tau\sum_{n=1}^{N}(6Q_g+12Q_fQ_{\bar{\alpha},T})\right)\left\|U_H^0\right\|_{H}+6\tau\sum_{n=1}^{N}\left\|b^n\right\|_H\\
&+12\tau Q_{\bar{\alpha},T}\sum_{n=1}^{N}\left\|f(U_H^0)\right\|_H+3\tau \left(Q_gQ_{T}+3Q_fQ_{\bar{\alpha},T}\right)\left\|U_H^{N}\right\|_{H}\\
&+6\tau \left(Q_g+Q_gQ_T+2Q_fQ_{\bar{\alpha},T}\right)\sum_{n=2}^{N}\left\|U_H^{n-1}\right\|_{H}.
\end{aligned}  
\end{equation}
Thus, when $\tau\le 1/(Q_g+6Q_gQ_T+18Q_fQ_{\bar{\alpha},T})$, the following estimate 
can be obtained
\begin{equation}\nonumber
\left\|U_{H}^{N}\right\|_{H}\le Q\tau\sum_{n=1}^{N}\left(\left\|U_H^0\right\|_{H}+\left\|f(U_H^0)\right\|_H+\left\|b^n\right\|_H\right)+Q\tau\sum_{n=1}^{N}\left\|U_H^{n-1}\right\|_{H}.
\end{equation}
At last, taking the initial condition \eqref{coarse-scheme-3} into consideration, 
in combination with the fact $|b^n_{i,j}|\le Q_{g}|\bar{u}_{0}(x_{i},y_{j})|$ for 
$(i,j)\in \omega_{H}$, we apply the discrete Gr\"onwall's inequality to derive 
\begin{equation}\nonumber
\left\|U_{H}^{N}\right\|_{H}\le Q\left(\left\|U_H^0\right\|_{H}+\left\|f(U_H^0)\right\|_H+\left\|\p_t U_H^0\right\|_H\right).
\end{equation}
The proof of this theorem is complete.
\end{proof}

\begin{theorem}
If the temporal stepsize $\tau$ is sufficiently small, the numerical solution 
of linear compact difference system \eqref{fine-scheme-1}-\eqref{fine-scheme-4} 
on the fine grid is stable, and there exists a positive constant 
$Q=Q(T,\bar{\alpha},Q_g,Q_f)$ such that
\begin{equation}\nonumber
\left\|U_{h}^{n}\right\|_{h}\le Q\left(\left\|u_0\right\|_{h}+\left\|f(u_0)\right\|_h+\left\|\bar{u}_0\right\|_h+\left\|u_0\right\|_{H}+\left\|f(u_0)\right\|_H+\left\|\bar{u}_0\right\|_H\right),\quad 1\le n\le N.
\end{equation}
\end{theorem}
\begin{proof}
First, when $n=1$, taking the inner product $(\cdot , \cdot)_{h}$ of equation 
\eqref{fine-scheme-1} with $2\tau U_h^1$, we have
\begin{equation}\nonumber
2c_{0}\tau\left(\mathcal{A}_{h}\delta_{t}U_{h}^{1},U_{h}^{1}\right)_{h}-2\tau\left(\mathfrak{I}_{\bar{\alpha}}^1\Lambda_{h} U_{h}^{1},U_{h}^{1}\right)_{h}=2\lambda_{1,1}\tau\left(\mathcal{A}_{h}F^{1},U_{h}^{1}\right)_{h}+2\tau\left(\mathcal{A}_{h}b^{1},U_{h}^{1}\right)_{h},
\end{equation}
which, together with Cauchy-Schwarz inequality, implies 
\begin{equation}\label{e4.14}
\begin{aligned}
c_{0}\left\|U_h^1\right\|^2_{\mathcal{A}_h}-2\tau\left(\mathfrak{I}_{\bar{\alpha}}^1\Lambda_{h} U_{h}^{1},U_{h}^{1}\right)_{h}
\le c_{0}\left\|U_h^0\right\|^2_{\mathcal{A}_h}+2\tau\left(\left\|\mathcal{A}_{h}b^1\right\|_h+Q_{\bar{\alpha},T}\left\|\mathcal{A}_{h}F^1\right\|_h\right)\left\|U_h^1\right\|_h.
\end{aligned}  
\end{equation}
Then, taking the inner product $(\cdot , \cdot)_{h}$ of equation 
\eqref{fine-scheme-2} with $2\tau U_h^{n-\frac{1}{2}}$ and summing up 
for $n$ from $2$ to $m$ for $2\le m\le N$, we get 
\begin{equation}\label{e4.15}
\begin{aligned}
&2c_{0}\tau\sum_{n=2}^{m}\left(\mathcal{A}_{h}\delta_{t}U_{h}^{n},U_{h}^{n-\frac{1}{2}}\right)_{h}-2\tau\sum_{n=2}^{m}\left(\mathfrak{I}_{\bar{\alpha}}^{n}\Lambda_{h} U_h^{n},U_{h}^{n-\frac{1}{2}}\right)_{h}\\
=&-2\tau\sum_{n=2}^{m}\sum_{k=1}^{n-1}\tilde{w}_{n-k}\left(\mathcal{A}_{h}\delta_{t}U_{h}^{k},U_{h}^{n-\frac{1}{2}}\right)_{h}+\tau\sum_{n=2}^{m}\lambda_{n,n}\left(\mathcal{A}_{h}F^{n},U_{h}^{n-\frac{1}{2}}\right)_{h}\\
&+2\tau\sum_{n=2}^{m}\sum_{k=1}^{n-1}\widetilde{\lambda}_{n,k}\left(\mathcal{A}_{h}f(U_{h}^{k}),U_{h}^{n-\frac{1}{2}}\right)_{h}+2\tau\sum_{n=2}^{m}\left(\mathcal{A}_{h}b^{n},U_{h}^{n-\frac{1}{2}}\right)_{h}.
\end{aligned}
\end{equation}
Similar to the derivation process of \eqref{e4.4} and \eqref{e4.5}, in combination with 
Cauchy-Schwarz inequality, it is easy to transform \eqref{e4.15} into 
\begin{equation}\label{e4.16}
\begin{aligned}
&c_{0}\|U_{h}^{m}\|_{\mathcal A_h}^{2}-2\tau\sum_{n=2}^{m}\left(\mathfrak{I}_{\bar{\alpha}}^{n}\Lambda_{h} U_h^{n},U_{h}^{n-\frac{1}{2}}\right)_{h}\\
\le &c_{0}\left\|U_h^1\right\|^2_{\mathcal{A}_h}+2\tau\sum_{n=2}^{m}\left\|\mathcal{A}_{h}b^n\right\|_h\left\|U_{h}^{n-\frac{1}{2}}\right\|_{h}+2\tau Q_{\bar{\alpha},T}\sum_{n=2}^{m}\left\|\mathcal{A}_{h}F^n\right\|_h\left\|U_{h}^{n-\frac{1}{2}}\right\|_{h}\\
&+\sum_{n=2}^{m}|\tilde{w}_{1}|\left\|U_h^{n-1}\right\|_{h}\left(\left\|U_h^{n}\right\|_{h}+\left\|U_h^{n-1}\right\|_{h}\right)+\sum_{n=2}^{m}|\tilde{w}_{n-1}|\left\|U_h^{0}\right\|_{h}\left(\left\|U_h^{n}\right\|_{h}+\left\|U_h^{n-1}\right\|_{h}\right)\\
&+\sum_{n=3}^{m}\sum_{k=2}^{n-1}|\tilde{w}_{k}-\tilde{w}_{k-1}|\left\|U_{h}^{n-k}\right\|_{h}\left(\left\|U_h^{n-1}\right\|_{h}+\left\|U_h^{n}\right\|_{h}\right)\\
&+\tau \sum_{n=2}^m\sum_{k=1}^{n-1}\left[\widetilde{\lambda}_{n,k}\left(Q_f\left\|U_h^{k}\right\|+Q_f\left\|U_h^0\right\|_h+\left\|f(U_h^0)\right\|_h\right)\right]\left(\left\|U_h^{n-1}\right\|_h+\left\|U_h^{n}\right\|_h\right).
\end{aligned}  
\end{equation}
in which, we utilized the inequality
\begin{equation}\nonumber
\begin{aligned}
&2\tau\sum_{n=2}^{m}\sum_{k=1}^{n-1}\widetilde{\lambda}_{n,k}\left(\mathcal{A}_{h}f(U_{h}^{k}),U_{h}^{n-\frac{1}{2}}\right)_{h}\\
\le &2\tau\sum_{n=2}^{m}\sum_{k=1}^{n-1}\widetilde{\lambda}_{n,k}\left\|\mathcal{A}_{h}f(U_{h}^{k})\right\|_{h}\left\|U_{h}^{n-\frac{1}{2}}\right\|_{h}\\
\leq&\tau \sum_{n=2}^m\sum_{k=1}^{n-1}\left[\widetilde{\lambda}_{n,k}\left(Q_f\left\|U_h^{k}-U_h^0\right\|_h+\left\|f(U_h^0)\right\|_h\right)\right]\left(\left\|U_h^{n-1}\right\|_h+\left\|U_h^{n}\right\|_h\right).
\end{aligned}  
\end{equation}
Next, from a similar process of \eqref{positive}, putting \eqref{e4.14} and 
\eqref{e4.16} together leads to 
\begin{equation}\nonumber
\begin{aligned}
&c_{0}\|U_{h}^{m}\|_{\mathcal A_h}^{2}-2\tau\left(\left\|\mathcal{A}_{h}b^1\right\|_h+Q_{\bar{\alpha},T}\left\|\mathcal{A}_{h}F^1\right\|_h\right)\left\|U_h^1\right\|_h\\
\le &c_{0}\left\|U_h^0\right\|^2_{\mathcal{A}_h}+2\tau\sum_{n=2}^{m}\left\|\mathcal{A}_{h}b^n\right\|_h\left\|U_{h}^{n-\frac{1}{2}}\right\|_{h}+2\tau Q_{\bar{\alpha},T}\sum_{n=2}^{m}\left\|\mathcal{A}_{h}F^n\right\|_h\left\|U_{h}^{n-\frac{1}{2}}\right\|_{h}\\
&+\sum_{n=2}^{m}|\tilde{w}_{1}|\left\|U_h^{n-1}\right\|_{h}\left(\left\|U_h^{n}\right\|_{h}+\left\|U_h^{n-1}\right\|_{h}\right)+\sum_{n=2}^{m}|\tilde{w}_{n-1}|\left\|U_h^{0}\right\|_{h}\left(\left\|U_h^{n}\right\|_{h}+\left\|U_h^{n-1}\right\|_{h}\right)\\
&+\sum_{n=3}^{m}\sum_{k=2}^{n-1}|\tilde{w}_{k}-\tilde{w}_{k-1}|\left\|U_{h}^{n-k}\right\|_{h}\left(\left\|U_h^{n-1}\right\|_{h}+\left\|U_h^{n}\right\|_{h}\right)\\
&+\tau \sum_{n=2}^m\sum_{k=1}^{n-1}\left[\widetilde{\lambda}_{n,k}\left(Q_f\left\|U_h^{k}\right\|+Q_f\left\|U_h^0\right\|_h+\left\|f(U_h^0)\right\|_h\right)\right]\left(\left\|U_h^{n-1}\right\|_h+\left\|U_h^{n}\right\|_h\right).
\end{aligned}  
\end{equation}
Therefore, choosing a suitable $\mathcal{J}$ such that 
$\|U_{h}^{\mathcal{J}}\|_{h}=\max\limits_{1\leq n\leq N}\|U_{h}^{n}\|_{h}$, 
we further obtain
\begin{equation}\label{e4.17}
\begin{aligned}
\frac{c_{0}}{3}\|U_{h}^{\mathcal{J}}\|_{h}
\le &c_{0}\left\|U_h^0\right\|_{h}+2\sum_{n=2}^{\mathcal{J}}|\tilde{w}_{1}|\left\|U_h^{n-1}\right\|_{h}+\sum_{n=3}^{\mathcal{J}}\sum_{k=2}^{n-1}|\tilde{w}_{k}-\tilde{w}_{k-1}|\left(\left\|U_h^{n-1}\right\|_{h}+\left\|U_h^{n}\right\|_{h}\right)\\
&+2\sum_{n=2}^{\mathcal{J}}|\tilde{w}_{n-1}|\left\|U_h^{0}\right\|_{h}+\tau Q_f\sum_{n=2}^{\mathcal{J}}\sum_{k=1}^{n-1}\widetilde{\lambda}_{n,k}\left(\left\|U_h^{n-1}\right\|_h+\left\|U_h^{n}\right\|_h\right)\\
&+2\tau\sum_{n=2}^{\mathcal{J}}\sum_{k=1}^{n-1}\widetilde{\lambda}_{n,k}\left(Q_f\left\|U_h^0\right\|_h+\left\|f(U_h^0)\right\|_h\right)+2\tau Q_{\bar{\alpha},T}\sum_{n=1}^{\mathcal{J}}\left\|F^n\right\|_h+2\tau\sum_{n=1}^{\mathcal{J}}\left\|b^n\right\|_h.
\end{aligned}  
\end{equation}
According to the globally boundedness condition of $f$, it is clear that 
\begin{equation}\label{e4.18}
\begin{aligned}
\tau\sum_{n=1}^{\mathcal{J}}\left\|F^n\right\|_h\le&\tau\sum_{n=1}^{\mathcal{J}}\left(\left\|f(\Pi_{H}U_{H}^{n})\right\|_{h}+\left\|f^{\prime}(\Pi_{H}U_{H}^{n})U_{h}^{n}\right\|_{h}+\left\|f^{\prime}(\Pi_{H}U_{H}^{n})\Pi_{H}U_{H}^{n}\right\|_{h}\right)\\
\le &\tau\sum_{n=1}^{\mathcal{J}}\left(\|f(\Pi_HU_H^n)\|_h+Q_f\|U_h^n\|_h+48Q_f\|U_H^n\|_H\right).
\end{aligned}
\end{equation}
Substituting \eqref{lam3}, \eqref{wpro1}-\eqref{wpro4} and \eqref{e4.18} into \eqref{e4.17}, 
we have 
\begin{equation}\nonumber
\begin{aligned}
\left(1-\frac{Q_g}{2}\tau\right)\|U_{h}^{N}\|_{h}
\le &\left(3+6\tau\sum_{n=1}^{N}\left(Q_g+Q_fQ_{\bar{\alpha},T}\right)\right)\left\|U_h^0\right\|_{h}+3\tau\left(Q_gQ_{T}+3Q_{f}Q_{\bar{\alpha},T}\right)\left\|U_{h}^{N}\right\|_{h}\\
&+6\tau\left(Q_g+Q_gQ_T+2Q_fQ_{\bar{\alpha},T}\right)\sum_{n=2}^N\left\|U_h^{n-1}\right\|_h\\
&+6\tau\sum_{n=1}^{N}\left(Q_{\bar{\alpha},T}\|f(\Pi_HU_H^n)\|_h+Q_{\bar{\alpha},T}\left\|f(U_h^0)\right\|_h+48Q_fQ_{\bar{\alpha},T}\|U_H^n\|_H+\left\|b^n\right\|_h\right).
\end{aligned}  
\end{equation}
Thus, when $\tau\le 1/(Q_{g}+6Q_gQ_T+18Q_fQ_{\bar{\alpha},T})$, the following estimate 
can be obtained by Grönwall's inequality
\begin{equation}\nonumber
\left\|U_{h}^{N}\right\|_{h}\le Q\tau\sum_{n=1}^{N}\left(\left\|U_h^0\right\|_{h}+\left\|f(U_h^0)\right\|_h+\|f(\Pi_HU_H^n)\|_h+\|U_H^n\|_H+\left\|b^n\right\|_h\right),
\end{equation}
which, together with the initial condition \eqref{fine-scheme-3} and the fact 
$|b^n_{i,j}|\le Q_g|\bar{u}_{0}(x_{i},y_{j})|$ for $(i,j)\in \omega_{h}$, implies
\begin{equation}\label{stabonfine}
  \left\|U_{h}^{N}\right\|_{h}\le Q\left[\left\|U_h^0\right\|_{h}+\left\|f(U_h^0)\right\|_h+\left\|\p_tU_h^{0}\right\|_h+\tau\sum_{n=1}^{N}\left(\|f(\Pi_HU_H^n)\|_h+\|U_H^n\|_H\right)\right].
\end{equation}
From Theorem \ref{staboncoarse}, it holds that
\begin{equation}\nonumber
\left\|U_{H}^{n}\right\|_{H}\le Q\left(\left\|U_H^0\right\|_{H}+\left\|f(U_H^0)\right\|_H+\left\|\p_t U_H^0\right\|_H\right), 
\end{equation}
and that
\begin{align*}
\left\|f(\Pi_HU_H^n)\right\|_h&=\left\|f(\Pi_HU_H^n)-f(U_h^0)+f(U_h^0)\right\|_h\\
&\le Q_f\left\|\Pi_HU_H^n-U_h^0\right\|_h+\left\|f(U_h^0)\right\|_h\\
&\le 48Q_f\left\|U_{H}^{n}\right\|_H + Q_f\left\|U_h^0\right\|_h+ \left\|f(U_h^0)\right\|_h\\
&\le Q\left(\left\|U_H^0\right\|_{H}+\left\|f(U_H^0)\right\|_H+\left\|\p_t U_H^0\right\|_H
+\left\|U_h^0\right\|_h+ \left\|f(U_h^0)\right\|_h\right).
\end{align*}
We combine the above two estimates with \eqref{stabonfine} to complete the proof.
\end{proof}

\subsection{Convergence}
Here, we first establish the convergence results of the STG compact difference scheme on the coarse grid using the energy method.
\begin{theorem}
Let $u_{i,j}^n$ and $[U_H]_{i,j}^{n}$ be the solutions of \eqref{e3.12} and \eqref{coarse-scheme-1}-\eqref{coarse-scheme-4} on the coarse grid, respectively. Under the conditions in Remark \ref{regularity}, if the 
step sizes $\tau$, $H_x$ and $H_y$ are sufficiently small, we have 
\begin{equation*}
\left\|u^{n}-U_{H}^{n}\right\|_{H}\le Q\left(\tau^{2}+H_x^4+H_y^4\right),\qquad \text{for}~1\le n\le N.
\end{equation*}
Here, the constant $Q$ is some constant depending on $T,\bar{\alpha},Q_{g},Q_f$, $Q_L$, and $Q_L=\sqrt{(r_x-l_x)(r_y-l_y)}$.
\end{theorem}
\begin{proof}
Let $[e_H]_{i,j}^{n}=u_{i,j}^{n}-[U_H]_{i,j}^{n}$ for $(i,j)\in\bar{\omega}_H$ 
and $0\le n\le N$. Subtracting \eqref{coarse-scheme-1} and \eqref{coarse-scheme-2} 
from \eqref{e3.12} on the coarse grid, we have 
\begin{equation}\label{conver-1}
c_{0}[\mathcal{A}_{H}\delta_{t}e_{H}]_{i,j}^{1}-\lambda_{1,1}[\Lambda_{H} e_{H}]_{i,j}^{1}=\lambda_{1,1}\mathcal{A}_{H}\left(f(u_{i,j}^{1})-f([U_{H}]_{i,j}^{1})\right)+(R)_{i,j}^{1},
\end{equation}
and for $2\le n\le N$
\begin{equation}\label{conver-2}
\begin{aligned}
&c_{0}[\mathcal{A}_{H}\delta_{t}e_{H}]_{i,j}^{n}+\sum_{k=1}^{n-1}\tilde{w}_{n-k}[\mathcal{A}_{H}\delta_{t}e_{H}]_{i,j}^{k}-\mathfrak{I}_{\bar{\alpha}}^{n}[\Lambda_{H} e_H]_{i,j}^{n}\\
=&\frac{\lambda_{n,n}\mathcal{A}_{H}\left(f(u_{i,j}^{n})-f([U_{H}]_{i,j}^{n})\right)}{2}+\sum_{k=1}^{n-1}\widetilde{\lambda}_{n,k}\mathcal{A}_{H}\left(f(u_{i,j}^{k})-f([U_{H}]_{i,j}^{k})\right)+(R)_{i,j}^{n},
\end{aligned}
\end{equation}
with $(i,j)\in\omega_{H}$. Besides, it holds that $[e_{H}]_{i,j}^0=0$ for $(i,j)\in\omega_{H}$.
Then, taking the inner product $(\cdot , \cdot)_{H}$ of equation 
\eqref{conver-1} with $2\tau e_H^1$, we have
\begin{equation}\label{e4.21}
\begin{aligned}
2c_{0}\tau\left(\mathcal{A}_{H}\delta_{t}e_{H}^{1},e_{H}^{1}\right)_{H}-2\lambda_{1,1}\tau\left(\Lambda_{H} e_{H}^{1},e_{H}^{1}\right)_{H}
=2\lambda_{1,1}\tau\left(\mathcal{A}_{H}\left(f(u^{1})-f(U_{H}^{1})\right),e_{H}^{1}\right)_{H}+2\tau\left(R^1,e_{H}^{1}\right)_{H}.
\end{aligned}
\end{equation}
It is obvious that 
\begin{equation*}
2\tau\left(\mathcal{A}_{H}\delta_t e_H^1,e_H^1\right)_{H} \geq \|e_H^1\|^2_{\mathcal{A}_H}-\|e_H^0\|^2_{\mathcal{A}_H},
\end{equation*}
and 
\begin{equation*}
\left(\mathcal{A}_H\left(f(u^{1})-f(U_{H}^{1})\right),e_H^1\right)_{H}\le \left\|f(u^{1})-f(U_{H}^{1})\right\|_H\left\|e_H^1\right\|_H\le Q_{f}\left\|e_H^1\right\|_H^2.
\end{equation*}
Then, \eqref{e4.21} turns into 
\begin{equation}\label{e4.22}
c_{0}\|e_H^1\|^2_{\mathcal{A}_H}-2\tau\left(\mathfrak{I}_{\bar{\alpha}}^1\Lambda_{H} e_{H}^{1},e_{H}^{1}\right)_{H}\le 2\tau Q_{\bar{\alpha},T}Q_f\left\|e_H^1\right\|_H^2+2\tau\left\|R^1\right\|_H\left\|e_H^1\right\|_H.
\end{equation}
Next, taking the inner product $(\cdot,\cdot)_{H}$ of equation \eqref{conver-2} 
with $2\tau e_H^{n-\frac{1}{2}}$ and summing up 
for $n$ from $2$ to $m$ for $2\le m\le N$, we get 
\begin{equation}\label{e4.23}
\begin{aligned}
&2c_{0}\tau\sum_{n=2}^{m}\left(\mathcal{A}_{H}\delta_{t}e_{H}^{n},e_{H}^{n-\frac{1}{2}}\right)_{H}-2\tau\sum_{n=2}^{m}\left(\mathfrak{I}_{\bar{\alpha}}^{n}\Lambda_{H} e_H^{n},e_{H}^{n-\frac{1}{2}}\right)_{H}\\
=&-2\tau\sum_{n=2}^{m}\sum_{k=1}^{n-1}\tilde{w}_{n-k}\left(\mathcal{A}_{H}\delta_{t}e_{H}^{k},e_{H}^{n-\frac{1}{2}}\right)_{H}+2\tau\sum_{n=2}^{m}\left(R^n,e_{H}^{n-\frac{1}{2}}\right)_{H}\\
&+\tau\sum_{n=2}^{m}\lambda_{n,n}\left(\mathcal{A}_{H}\left(f(u^{n})-f(U_{H}^{n})\right),e_{H}^{n-\frac{1}{2}}\right)_{H}\\
&+2\tau\sum_{n=2}^{m}\sum_{k=1}^{n-1}\widetilde{\lambda}_{n,k}\left(\mathcal{A}_{H}\left(f(u^{k})-f(U_{H}^{k})\right),e_{H}^{n-\frac{1}{2}}\right)_{H}.
\end{aligned}
\end{equation}
Following a procedure similar to the derivation of \eqref{e4.4} and \eqref{e4.5}, and applying the Cauchy–Schwarz inequality, we transform \eqref{e4.23} into
\begin{equation}\label{e4.24}
\begin{aligned}
&c_{0}\|e_{H}^{m}\|_{\mathcal A_H}^{2}-c_{0}\left\|e_H^1\right\|^2_{\mathcal{A}_H}-2\tau\sum_{n=2}^{m}\left(\mathfrak{I}_{\bar{\alpha}}^{n}\Lambda_{H} e_H^{n},e_{H}^{n-\frac{1}{2}}\right)_{H}\\
\le &2\tau\sum_{n=2}^{m}\left\|R^n\right\|_H\left\|e_{H}^{n-\frac{1}{2}}\right\|_{H}+\sum_{n=2}^{m}|\tilde{w}_{1}|\left\|e_H^{n-1}\right\|_{H}\left(\left\|e_H^{n}\right\|_{H}+\left\|e_H^{n-1}\right\|_{H}\right)\\
&+\sum_{n=3}^{m}\sum_{k=2}^{n-1}|\tilde{w}_{k}-\tilde{w}_{k-1}|\left\|e_{H}^{n-k}\right\|_{H}\left(\left\|e_H^{n-1}\right\|_{H}+\left\|e_H^{n}\right\|_{H}\right)\\
&+\tau Q_f\sum_{n=2}^{m}\lambda_{n,n}\left\|e_H^{n}\right\|_{H}\left\|e_H^{n-\frac{1}{2}}\right\|_{H}+2\tau Q_f\sum_{n=2}^{m}\sum_{k=1}^{n-1}\widetilde{\lambda}_{n,k}\left\|e_H^{k}\right\|_{H}\left\|e_H^{n-\frac{1}{2}}\right\|_{H},
\end{aligned}  
\end{equation}
in which, we used the fact 
\begin{equation}\nonumber
\begin{aligned}
\left(\mathcal{A}_{H}\left(f(u^{k})-f(U_{H}^{k})\right),e_{H}^{n-\frac{1}{2}}\right)_{H}\le Q_f\left\|e_H^{k}\right\|_{H}\left\|e_H^{n-\frac{1}{2}}\right\|_{H}\quad \text{for}~1\le k\le m.
\end{aligned}  
\end{equation}
Next, performing a derivation similar to \eqref{positive}, putting \eqref{e4.22} and 
\eqref{e4.24} together leads to 
\begin{equation}\nonumber
\begin{aligned}
&c_{0}\|e_{H}^{m}\|_{\mathcal A_H}^{2}-2\tau Q_{\bar{\alpha},T}Q_f\left\|e_H^1\right\|_H^2-2\tau\left\|R^1\right\|_H\left\|e_H^1\right\|_H\\
\le &2\tau\sum_{n=2}^{m}\left\|R^n\right\|_H\left\|e_{H}^{n-\frac{1}{2}}\right\|_{H}+\sum_{n=2}^{m}|\tilde{w}_{1}|\left\|e_H^{n-1}\right\|_{H}\left(\left\|e_H^{n}\right\|_{H}+\left\|e_H^{n-1}\right\|_{H}\right)\\
&+\sum_{n=3}^{m}\sum_{k=2}^{n-1}|\tilde{w}_{k}-\tilde{w}_{k-1}|\left\|e_{H}^{n-k}\right\|_{H}\left(\left\|e_H^{n-1}\right\|_{H}+\left\|e_H^{n}\right\|_{H}\right)\\
&+\tau Q_f\sum_{n=2}^{m}\lambda_{n,n}\left\|e_H^{n}\right\|_{H}\left\|e_H^{n-\frac{1}{2}}\right\|_{H}+2\tau Q_f\sum_{n=2}^{m}\sum_{k=1}^{n-1}\widetilde{\lambda}_{n,k}\left\|e_H^{k}\right\|_{H}\left\|e_H^{n-\frac{1}{2}}\right\|_{H}.
\end{aligned}  
\end{equation}
Therefore, choosing a suitable $\mathcal{S}$ such that 
$\|e_{H}^{\mathcal{S}}\|_{H}=\max\limits_{1\leq n\leq N}\|e_{H}^{n}\|_{H}$, 
we further obtain
\begin{equation}\nonumber
\begin{aligned}
\frac{c_{0}}{3}\|e_{H}^{\mathcal{S}}\|_{H}\le &2\sum_{n=2}^{\mathcal{S}}|\tilde{w}_{1}|\left\|e_H^{n-1}\right\|_{H}+\sum_{n=3}^{\mathcal{S}}\sum_{k=2}^{n-1}|\tilde{w}_{k}-\tilde{w}_{k-1}|\left(\left\|e_H^{n-1}\right\|_{H}+\left\|e_H^{n}\right\|_{H}\right)\\
&+2\tau\sum_{n=1}^{\mathcal{S}}\left\|R^n\right\|_H+2\tau Q_f Q_{\bar{\alpha},T}\sum_{n=1}^{\mathcal{S}}\left\|e_H^{n}\right\|_{H}+2\tau Q_f\sum_{n=2}^{\mathcal{S}}\sum_{k=1}^{n-1}\widetilde{\lambda}_{n,k}\left\|e_H^{n-\frac{1}{2}}\right\|_{H},
\end{aligned}  
\end{equation}
which, together with \eqref{lam3}, \eqref{wpro1}-\eqref{wpro4}, implies
\begin{equation}\nonumber
\begin{aligned}
&\left(1-\frac{Q_g\tau}{2}\right)\|e_{H}^{N}\|_{H}\\
\le &6\tau Q_g\sum_{n=2}^{N}\left\|e_H^{n-1}\right\|_{H}+3\tau Q_gQ_T\sum_{n=3}^{N}\left(\left\|e_H^{n-1}\right\|_{H}+\left\|e_H^{n}\right\|_{H}\right)+6\tau Q_L\sum_{n=1}^{N}\left\|R^n\right\|_{H,\infty}\\
&+6\tau Q_f Q_{\bar{\alpha},T}\sum_{n=1}^{N}\left\|e_H^{n}\right\|_{H}+3\tau Q_fQ_{\bar{\alpha},T}\sum_{n=2}^{N}\left(\left\|e_H^{n-1}\right\|_{H}+\left\|e_H^{n}\right\|_{H}\right).
\end{aligned}  
\end{equation}
Thus, when $\tau\le 1/(Q_g+6Q_gQ_T+18Q_fQ_{\bar{\alpha},T})$, the following estimate 
holds
\begin{equation*}
\|e_{H}^{N}\|_{H}\le 12\tau\left(Q_g+Q_gQ_T+2Q_fQ_{\bar{\alpha},T}\right)\sum_{n=2}^{N}\left\|e_H^{n-1}\right\|_{H}+12\tau Q_L\sum_{n=1}^{N}\left\|R^n\right\|_{H,\infty}.
\end{equation*}
We then apply Gr\"{o}nwall's inequality and Lemma \ref{errorestimate} to yield
\begin{equation}\nonumber
\|e_{H}^{N}\|_{H}\le Q\left(\tau^{2}+H_x^4+H_y^4\right).
\end{equation}
We thus complete the proof.
\end{proof}

\begin{remark}\label{remark2}
In fact, by applying Lemmas \ref{inteesti1}--\ref{inteesti2}, the above theorem implies that
\begin{equation*}
\begin{aligned}
\left\|u^{n}-\Pi_{H}U_{H}^{n}\right\|_{h}&=\left\|u^{n}-\Pi_{H}u^{n}+\Pi_{H}u^{n}-\Pi_{H}U_{H}^{n}\right\|_{h}\\
&\le \left\|u^{n}-\Pi_{H}u^{n}\right\|_{h}+\left\|\Pi_{H}u^{n}-\Pi_{H}U_{H}^{n}\right\|_{h}\\
&\le Q_L\left\|u^{n}-\Pi_{H}u^{n}\right\|_{h,\infty}+48\left\|u^{n}-U_{H}^{n}\right\|_{H}\\
&\le Q\left(\tau^{2}+H_x^4+H_y^4\right).
\end{aligned}  
\end{equation*}
\end{remark}

\vskip 1mm 
At last, we shall give an error estimate for the linearized compact difference scheme
\eqref{fine-scheme-1}-\eqref{fine-scheme-4} on the fine grid. To this end, denote 
$[e_h]_{i,j}^{n}=u_{i,j}^{n}-[U_h]_{i,j}^{n}$ for $(i,j)\in\bar{\omega}_h$ and $0\le n\le N$.

\vskip 1mm 
\begin{theorem}
Let $u_{i,j}^n$ and $[U_h]_{i,j}^{n}$ be the solutions of \eqref{e3.12} and \eqref{fine-scheme-1}-\eqref{fine-scheme-4} on the fine grid, respectively. Under the conditions in Remark \ref{regularity}, if the step sizes $\tau$, $h_x$ and $h_y$ are sufficiently small, we have 
\begin{equation*}
\left\|u^{n}-U_{h}^{n}\right\|_{h}\le Q\left(\tau^{2}+h_x^4+h_y^4+H_4^4+H_y^4\right),\qquad \text{for}~1\le n\le N.
\end{equation*}
Here, the constant $Q$ is some constant depending on $T,\bar{\alpha},Q_g,Q_f$, $Q_L$.
\end{theorem}
\begin{proof}
subtracting \eqref{fine-scheme-1} and \eqref{fine-scheme-2} 
from \eqref{e3.12} on the fine grid, we obtain
\begin{equation}\label{conver-3}
c_{0}[\mathcal{A}_{h}\delta_{t}e_{h}]_{i,j}^{1}-\lambda_{1,1}[\Lambda_{h} e_{h}]_{i,j}^{1}=\lambda_{1,1}\mathcal{A}_{h}\left(f(u_{i,j}^{1})-F_{i,j}^{1}\right)+(R)_{i,j}^{1},
\end{equation}
and for $2\le n\le N$ 
\begin{equation}\label{conver-4}
\begin{aligned}
    &c_{0}[\mathcal{A}_{h}\delta_{t}e_{h}]_{i,j}^{n}+\sum_{k=1}^{n-1}\tilde{w}_{n-k}[\mathcal{A}_{h}\delta_{t}e_{h}]_{i,j}^{k}-\mathfrak{I}_{\bar{\alpha}}^{n}[\Lambda_h e_h]_{i,j}^{n}\\
    =&\frac{\lambda_{n,n}\mathcal{A}_{h}\left(f(u_{i,j}^{n})-F_{i,j}^{n}\right)}{2}+\sum_{k=1}^{n-1}\widetilde{\lambda}_{n,k}\mathcal{A}_{h}\left(f(u_{i,j}^{k})-f([U_{h}]_{i,j}^{k})\right)+(R)_{i,j}^{n},
\end{aligned}
\end{equation}
with $(i,j)\in\omega_{H}$. Besides, it holds that $[e_{h}]_{i,j}^0=0$ for $(i,j)\in\omega_{h}$. Then, taking the inner product $(\cdot,\cdot)_{h}$ of equation \eqref{conver-3} with $2\tau e_h^1$, we have
\begin{equation}\nonumber
2c_{0}\tau\left(\mathcal{A}_{h}\delta_{t}e_{h}^{1},e_{h}^{1}\right)_{h}-2\lambda_{1,1}\tau\left(\Lambda_h e_{h}^{1},e_{h}^{1}\right)_{h}
=2\lambda_{1,1}\tau\left(\mathcal{A}_{h}\left(f(u^{1})-F^{1}\right),e_{h}^{1}\right)_{h}+2\tau\left(R^1,e_{h}^{1}\right)_{h},
\end{equation}
which, together with Cauchy-Schwarz inequality, implies 
\begin{equation}\label{e4.34}
c_{0}\left\|e_h^1\right\|^2_{\mathcal{A}_h}-2\tau\left(\mathfrak{I}_{\bar{\alpha}}^1\Lambda_h e_{h}^{1},e_{h}^{1}\right)_{h}
\le 2\lambda_{1,1}\tau\left\|\mathcal{A}_{h}\left(f(u^{1})-F^{1}\right)\right\|_{h}\left\|e_h^1\right\|_h+2\tau Q_L\left\|R^1\right\|_{h,\infty}\left\|e_h^1\right\|_h.
\end{equation}
Then, taking the inner product $(\cdot,\cdot)_{h}$ of equation \eqref{conver-4} 
with $2\tau e_h^{n-\frac{1}{2}}$ and summing up 
for $n$ from $2$ to $m$ for $2\le m\le N$, we get 
\begin{equation}\label{e4.35}
\begin{aligned}
&2c_{0}\tau\sum_{n=2}^{m}\left(\mathcal{A}_{h}\delta_{t}e_{h}^{n},e_{h}^{n-\frac{1}{2}}\right)_{h}-2\tau\sum_{n=2}^{m}\left(\mathfrak{I}_{\bar{\alpha}}^{n}\Lambda_h e_h^{n},e_{h}^{n-\frac{1}{2}}\right)_{h}\\
=&-2\tau\sum_{n=2}^{m}\sum_{k=1}^{n-1}\tilde{w}_{n-k}\left(\mathcal{A}_{h}\delta_{t}e_{h}^{k},e_{h}^{n-\frac{1}{2}}\right)_{h}+\tau\sum_{n=2}^{m}\lambda_{n,n}\left(\mathcal{A}_{h}\left(f(u^{n})-F^n\right),e_{h}^{n-\frac{1}{2}}\right)_{h}\\
&+2\tau\sum_{n=2}^{m}\sum_{k=1}^{n-1}\widetilde{\lambda}_{n,k}\left(\mathcal{A}_{h}\left(f(u^{k})-f(U_{h}^{k})\right),e_{h}^{n-\frac{1}{2}}\right)_{h}+2\tau\sum_{n=2}^{m}\left(R^n,e_{h}^{n-\frac{1}{2}}\right)_{h}.
\end{aligned}
\end{equation}
Following a procedure similar to the derivation of \eqref{e4.4} and \eqref{e4.5}, and using the Cauchy–Schwarz inequality, \eqref{e4.35} can be readily transformed into 
\begin{equation}\label{e4.36}
\begin{aligned}
&c_{0}\|e_{h}^{m}\|_{\mathcal A_h}^{2}-c_{0}\left\|e_h^1\right\|^2_{\mathcal{A}_h}-2\tau\sum_{n=2}^{m}\left(\mathfrak{I}_{\bar{\alpha}}^{n}\Lambda_h e_h^{n},e_{h}^{n-\frac{1}{2}}\right)_{h}\\
\le &2\tau Q_L\sum_{n=2}^{m}\left\|R^n\right\|_{h,\infty}\left\|e_{h}^{n-\frac{1}{2}}\right\|_{h}+\sum_{n=2}^{m}|\tilde{w}_{1}|\left\|e_h^{n-1}\right\|_{h}\left(\left\|e_h^{n}\right\|_{h}+\left\|e_h^{n-1}\right\|_{h}\right)\\
&+\sum_{n=3}^{m}\sum_{k=2}^{n-1}|\tilde{w}_{k}-\tilde{w}_{k-1}|\left\|e_{h}^{n-k}\right\|_{h}\left(\left\|e_h^{n-1}\right\|_{h}+\left\|e_h^{n}\right\|_{h}\right)\\
&+\tau\sum_{n=2}^{m}\lambda_{n,n}\left\|\mathcal{A}_{h}\left(f(u^{n})-F^{n}\right)\right\|_{h}\left\|e_h^{n-\frac{1}{2}}\right\|_{h}+2\tau Q_f\sum_{n=2}^{m}\sum_{k=1}^{n-1}\widetilde{\lambda}_{n,k}\left\|e_h^{k}\right\|_{h}\left\|e_h^{n-\frac{1}{2}}\right\|_{h},
\end{aligned}  
\end{equation}
in which, we used the fact 
\begin{equation}\nonumber
\begin{aligned}
\left(\mathcal{A}_{h}\left(f(u^{k})-f(U_{h}^{k})\right),e_{h}^{n-\frac{1}{2}}\right)_{h}\le Q_f\left\|e_h^{k}\right\|_{h}\left\|e_h^{n-\frac{1}{2}}\right\|_{h}\quad \text{for}~1\le k\le m-1.
\end{aligned}  
\end{equation}
Next, from a similar process for \eqref{positive}, putting \eqref{e4.34} and 
\eqref{e4.36} together leads to
\begin{equation}\nonumber
\begin{aligned}
&c_{0}\|e_{h}^{m}\|_{\mathcal A_h}^{2}-2\lambda_{1,1}\tau\left\|\mathcal{A}_{h}\left(f(u^{1})-F^{1}\right)\right\|_{h}\left\|e_h^1\right\|_h-2\tau Q_L\left\|R^1\right\|_{h,\infty}\left\|e_h^1\right\|_h\\
\le &2\tau Q_L\sum_{n=2}^{m}\left\|R^n\right\|_{h,\infty}\left\|e_{h}^{n-\frac{1}{2}}\right\|_{h}+\sum_{n=2}^{m}|\tilde{w}_{1}|\left\|e_h^{n-1}\right\|_{h}\left(\left\|e_h^{n}\right\|_{h}+\left\|e_h^{n-1}\right\|_{h}\right)\\
&+\sum_{n=3}^{m}\sum_{k=2}^{n-1}|\tilde{w}_{k}-\tilde{w}_{k-1}|\left\|e_{h}^{n-k}\right\|_{h}\left(\left\|e_h^{n-1}\right\|_{h}+\left\|e_h^{n}\right\|_{h}\right)\\
&+\tau\sum_{n=2}^{m}\lambda_{n,n}\left\|\mathcal{A}_{h}\left(f(u^{n})-F^{n}\right)\right\|_{h}\left\|e_h^{n-\frac{1}{2}}\right\|_{h}+2\tau Q_f\sum_{n=2}^{m}\sum_{k=1}^{n-1}\widetilde{\lambda}_{n,k}\left\|e_h^{k}\right\|_{h}\left\|e_h^{n-\frac{1}{2}}\right\|_{h}.
\end{aligned}  
\end{equation}
Therefore, choosing a suitable $\mathcal{F}$ such that 
$\|e_{h}^{\mathcal{F}}\|_{h}=\max\limits_{1\leq n\leq N}\|e_{h}^{n}\|_{h}$, we have
\begin{equation}\nonumber
\begin{aligned}
\frac{c_0}{3}\|e_{h}^{\mathcal{F}}\|_{h}\le &2\tau Q_L\sum_{n=1}^{\mathcal{F}}\left\|R^n\right\|_{h,\infty}+2\sum_{n=2}^{\mathcal{F}}|\tilde{w}_{1}|\left\|e_h^{n-1}\right\|_{h}+2\tau\sum_{n=1}^{\mathcal{F}}\lambda_{n,n}\left\|f(u^{n})-F^{n}\right\|_{h}\\
&+\sum_{n=3}^{\mathcal{F}}\sum_{k=2}^{n-1}|\tilde{w}_{k}-\tilde{w}_{k-1}|\left(\left\|e_h^{n-1}\right\|_{h}+\left\|e_h^{n}\right\|_{h}\right)+2\tau Q_f\sum_{n=2}^{\mathcal{F}}\sum_{k=1}^{n-1}\widetilde{\lambda}_{n,k}\left\|e_h^{n-\frac{1}{2}}\right\|_{h}.
\end{aligned}  
\end{equation}
At this point, our objective is to discuss the estimate of the third term on the right 
hand side. To this end, we apply the mean value theorem to obtain 
\begin{equation*}
  f(u_{i,j}^n)=f([\Pi_HU_H]_{i,j}^n)+f'(\eta_{i,j}^n)(u_{i,j}^n-[\Pi_HU_H]_{i,j}^n),
\end{equation*}
in which $\eta_{i,j}^{n}\in\left(\min\{u_{i,j}^{n},[\Pi_{H}U_{H}]_{i,j}^{n}\},\max\{u_{i,j}^{n},[\Pi_{H}U_{H}]_{i,j}^{n}\}\right)$ for $(i,j)\in\omega_h$ and $1\le n\le N$. In this case, we get 
\begin{equation*}
\begin{aligned}
\left\|f(u^{n})-F^{n}\right\|_{h}
=&\left\|\left(f'(\eta^{n})-f'(\Pi_{H}U_{H}^{n})\right)(u^{n}-\Pi_{H}U_{H}^{n})+f'(\Pi_{H}U_{H}^{n})(u^{n}-U_{h}^{n})\right\|_{h} \\
\le& 2Q_f\left\|u^{n}-\Pi_{H}U_{H}^{n}\right\|_{h} + Q_f\left\|e_h^n\right\|_{h},
\end{aligned}
\end{equation*}
which, together with \eqref{lam3} and \eqref{wpro1}-\eqref{wpro4}, implies
\begin{equation}\nonumber
\begin{aligned}
\left(1-\frac{Q_g\tau}{2}\right)\|e_{h}^{N}\|_{h}
\le &6\tau Q_g\sum_{n=2}^{N}\left\|e_h^{n-1}\right\|_{h}+3\tau Q_gQ_T\sum_{n=3}^{N}\left(\left\|e_h^{n-1}\right\|_{h}+\left\|e_h^{n}\right\|_{h}\right)+6\tau Q_L\sum_{n=1}^{N}\left\|R^n\right\|_{h,\infty}\\
&+6\tau Q_f Q_{\bar{\alpha},T}\sum_{n=1}^{N}\left(\left\|e_h^{n}\right\|_{h}+2\left\|u^{n}-\Pi_{H}U_{H}^{n}\right\|_{h}\right)
+3\tau Q_fQ_{\bar{\alpha},T}\sum_{n=2}^{N}\left(\left\|e_h^{n-1}\right\|_{h}+\left\|e_h^{n}\right\|_{h}\right).
\end{aligned}  
\end{equation}
Thus, when $\tau\le 1/(Q_g+6Q_gQ_T+18Q_fQ_{\bar{\alpha},T})$, the following estimate 
can be obtained
\begin{equation*}
\begin{aligned}
\|e_{h}^{N}\|_{h}\le &12\tau\left(Q_g+Q_gQ_T+2Q_fQ_{\bar{\alpha},T}\right)\sum_{n=2}^{N}\left\|e_h^{n-1}\right\|_{h}\\
&+12\tau Q_L\sum_{n=1}^{N}\left\|R^n\right\|_{h,\infty}+24\tau Q_f Q_{\bar{\alpha},T}\sum_{n=1}^{N}\left\|u^{n}-\Pi_{H}U_{H}^{n}\right\|_{h}.
\end{aligned}
\end{equation*}
Lemma \ref{errorestimate}, the estimates in Remark \ref{remark2} and Gr\"{o}nwall's inequality give
\begin{equation}\nonumber
\|e_{h}^{N}\|_{h}\le Q\left(\tau^{2}+h_x^4+h_y^4+H_x^4+H_y^4\right),
\end{equation}
which completes the proof of the theorem.
\end{proof}

\section{Numerical experiments}\label{sec5}
In this section, we perform numerical experiments to assess the efficiency and accuracy of the STG compact difference algorithm. For comparison, we also present results from the standard nonlinear compact difference scheme \eqref{standard fine} alongside those of the proposed STG scheme \eqref{coarse-scheme-1}--\eqref{fine-scheme-4}. All experiments are performed on a computer with Windows 11(64-bit) PC-Inter(R) Core(TM) i5-12400 CPU 2.50 GHz, 64.0 GB of RAM, using MATLAB R2021a. Below we take $\Omega=(0,1)\times (0,1)$, $h=h_x=h_y=\frac{1}{M}$ and $T=1$, with $M=M_h=M_{h,x}=M_{h,y}$.

To assess the convergence of proposed methods with unknown exact solutions, we apply the two-mesh principle \cite[p.~107]{Far} to define errors and temporal convergence rates as
\begin{equation*}
E(\tau,h)=\max\limits_{0\le n\le N}\sqrt{h^{2}\sum_{i=1}^{M-1}\sum_{j=1}^{M-1}\left|U_{i,j}^{n}(\tau,h)-U_{i,j}^{2n}(\frac{\tau}{2},h)\right|^{2}},\quad Rate^{\tau}=\log_2\left(\frac{E(2\tau,h)}{E(\tau,h)}\right),
\end{equation*}
and corresponding spatial errors and convergence rates as
\begin{equation*}
S(\tau,h)=\sqrt{h^{2}\sum_{i=1}^{M-1}\sum_{j=1}^{M-1}\left|U_{i,j}^{N}(\tau,h)-U_{2i,2j}^{N}(\tau,\frac{h}{2})\right|^{2}},\quad Rate^{h}=\log_2\left(\frac{S(\tau,2h)}{S(\tau,h)}\right),
\end{equation*}
with $U_{i}^{n}(\tau,h)$ being the numerical solution when time step size is $\tau$ 
and the spatial mesh size is $h$. For clarity, we denote by $M_H$ ($M_H = M_{H,x} = M_{H,y}$) and $M_h$ the numbers of coarse and fine mesh points in the two-grid scheme, respectively, while $M_h$ also represents the number of mesh points in the standard compact difference scheme.

\vskip 1mm
\textbf{Example.} We here consider the model 
\eqref{VtFDEs}-\eqref{bc} with $u_{0}(x,y)=\sin(\pi x)\sin(\pi y)$ and
$\bar{u}_{0}(x,y)=\sin(2\pi x)\sin(2\pi y)$. Meanwhile, to satisfy the very weak conditions $\alpha'(0)= 0$, we choose
$\alpha(t)=\alpha_{0}+\frac{1}{11}t^{2}$ and limit $1<\alpha_{0}\le 1.9$ such that
$\alpha(t)\in (1,2)$. Then, we focus on the nonlinear terms in the following two cases:
\begin{enumerate}[label=(\Roman*)]
    \item $f(u) = \frac{u^2}{1+u^2} - u$, satisfying $|f'(u)|\leq Q_f$;
    \item $f(u) = u - u^3$, for which $|f'(u)|$ is not globally bounded.
\end{enumerate}

\begin{table}\footnotesize\centering
\caption{Discrete errors, temporal convergence rates, and CPU time (seconds) with $M_h=M_H^2=64$ for Case (I).}
\label{table1}  
\begin{tabular}{ccccccccc}
\hline\noalign{\smallskip}
\multirow{2}{*}{$\alpha_{0}$} &\multirow{2}{*}{$N$} &\multicolumn{3}{c}{STG algorithm} & &  \multicolumn{3}{c}{Nonlinear algorithm} \\
\cline{3-5} \cline{7-9}
& & $E(\tau,h)$ & $Rate^{\tau}$ & CPU(s) & & $E(\tau,h)$ & $Rate^{\tau}$ & CPU(s)\\
\noalign{\smallskip}\hline\noalign{\smallskip}
       & 128  & 3.4832e-03 & *    & 1.76   & & 3.4832e-03 & *    & 2.15    \\
       & 256  & 8.5709e-04 & 2.02 & 3.86   & & 8.5709e-04 & 2.02 & 5.61    \\
$1.20$ & 512  & 2.0515e-04 & 2.06 & 9.59   & & 2.0515e-04 & 2.06 & 17.29   \\
       & 1024 & 4.8961e-05 & 2.07 & 26.59  & & 4.8961e-05 & 2.07 & 60.03   \\
       & 2048 & 1.1850e-05 & 2.05 & 81.64  & & 1.1850e-05 & 2.05 & 221.83  \\
\noalign{\smallskip}\hline\noalign{\smallskip}
       & 128  & 8.9653e-04 & *	   & 1.68   & & 8.9653e-04 & *	  & 2.11    \\
       & 256  & 2.0331e-04 & 2.14 & 3.76   & & 2.0331e-04 & 2.14 & 5.64    \\
$1.50$ & 512  & 4.8144e-05 & 2.08 & 9.24   & & 4.8144e-05 & 2.08 & 16.83   \\
       & 1024 & 1.1619e-05 & 2.05 & 26.04  & & 1.1619e-05 & 2.05 & 58.74  \\
       & 2048 & 2.8374e-06 & 2.03 & 81.38  & & 2.8374e-06 & 2.03 & 220.95  \\
\noalign{\smallskip}\hline\noalign{\smallskip}
       & 128  & 6.2986e-04 & *    & 2.08   & & 6.2986e-04 & *    & 2.54    \\
       & 256  & 1.5207e-04 & 2.05 & 4.62   & & 1.5207e-04 & 2.05 & 6.26   \\
$1.80$ & 512  & 3.7236e-05 & 2.03 & 10.99  & & 3.7236e-05 & 2.03 & 17.98   \\
       & 1024 & 9.1934e-06 & 2.02 & 29.40  & & 9.1934e-06 & 2.02 & 62.03  \\
       & 2048 & 2.2817e-06 & 2.01 & 87.79  & & 2.2817e-06 & 2.01 & 227.15  \\
\noalign{\smallskip}\hline
\end{tabular}
\end{table}

\begin{table}\footnotesize\centering
\caption{Discrete errors, spatial convergence rates, and CPU time (seconds) with $N=256$ and $J=4$ for Case (I).}
\label{table2}  
\begin{tabular}{cccccccccc}
\hline\noalign{\smallskip}
\multirow{2}{*}{$\alpha_{0}$} &\multirow{2}{*}{$M_H$} &\multirow{2}{*}{$M_h$} &\multicolumn{3}{c}{STG algorithm} & &  \multicolumn{3}{c}{Nonlinear algorithm} \\
\cline{4-6} \cline{8-10}
& & &   $S(\tau,h)$    & $Rate^{h}$ & CPU(s) & & $S(\tau,h)$    & $Rate^{h}$ & CPU(s) \\
\noalign{\smallskip}\hline\noalign{\smallskip}
       & 4 & 16 & 6.6089e-06 & *    & 0.41  & & 6.6316e-06 & *    & 0.37  \\
       & 8 & 32 & 4.0673e-07 & 4.02 & 1.04  & & 4.0655e-07 & 4.03 & 1.28  \\
$1.30$ & 16& 64 & 2.5289e-08 & 4.01 & 4.01  & & 2.5288e-08 & 4.01 & 5.64  \\
       & 32& 128& 1.5786e-09 & 4.00 & 16.84 & & 1.5786e-09 & 4.00 & 26.01 \\
       & 64& 256& 9.8634e-11 & 4.00 & 78.76 & & 9.8634e-11 & 4.00 & 114.53\\
\noalign{\smallskip}\hline\noalign{\smallskip}
       & 4 & 16 & 4.8238e-06 & *    & 0.38  & & 4.6728e-06 & *	   & 0.35  \\
       & 8 & 32 & 2.8932e-07 & 4.06 & 1.08  & & 2.8931e-07 & 4.01 & 1.38  \\
$1.60$ & 16& 64 & 1.8033e-08 & 4.00 & 4.31  & & 1.8032e-08 & 4.00 & 6.26  \\
       & 32& 128& 1.1262e-09 & 4.00 & 18.21 & & 1.1262e-09 & 4.00 & 29.72 \\
       & 64& 256& 7.0358e-11 & 4.00 & 85.29 & & 7.0355e-11 & 4.00 & 132.40\\
\noalign{\smallskip}\hline\noalign{\smallskip}
       & 4 & 16 & 1.6245e-04 & *    & 10.07 & & 1.6244e-04 & *    & 9.96 \\
       & 8 & 32 & 9.9627e-06 & 4.03 & 10.72 & & 9.9642e-06 & 4.03 & 10.86 \\
$1.90$ & 16& 64 & 6.2017e-07 & 4.01 & 13.81 & & 6.2017e-07 & 4.01 & 15.11 \\
       & 32& 128& 3.8721e-08 & 4.00 & 26.98 & & 3.8721e-08 & 4.00 & 35.22 \\
       & 64& 256& 2.4191e-09 & 4.00 & 90.42 & & 2.4191e-09 & 4.00 & 121.42\\
\noalign{\smallskip}\hline
\end{tabular}
\end{table}

\begin{table}\footnotesize\centering
\caption{Discrete errors, temporal convergence rates, and CPU time (seconds) with $M_h=M_H^2=64$ for Case (II).}
\label{table3}  
\begin{tabular}{ccccccccc}
\hline\noalign{\smallskip}
\multirow{2}{*}{$\alpha_{0}$} &\multirow{2}{*}{$N$} &\multicolumn{3}{c}{STG algorithm} & &  \multicolumn{3}{c}{Nonlinear algorithm} \\
\cline{3-5} \cline{7-9}
& & $E(\tau,h)$ & $Rate^{\tau}$ & CPU(s) & & $E(\tau,h)$ & $Rate^{\tau}$ & CPU(s)\\
\noalign{\smallskip}\hline\noalign{\smallskip}
       & 128  & 3.3151e-03 & *    & 1.91   & & 3.3151e-03 & *    & 3.22    \\
       & 256  & 8.2075e-04 & 2.01 & 4.22   & & 8.2075e-04 & 2.01 & 9.61    \\
$1.20$ & 512  & 1.9683e-04 & 2.06 & 10.54  & & 1.9683e-04 & 2.06 & 32.44   \\
       & 1024 & 4.7048e-05 & 2.06 & 29.58  & & 4.7048e-05 & 2.06 & 118.21  \\
       & 2048 & 1.1391e-05 & 2.05 & 93.90  & & 1.1391e-05 & 2.05 & 449.01  \\
\noalign{\smallskip}\hline\noalign{\smallskip}
       & 128  & 8.5833e-04 & *	   & 1.69   & & 8.5833e-04 & *	  & 3.41    \\
       & 256  & 1.9796e-04 & 2.12 & 3.81   & & 1.9796e-04 & 2.12 & 9.94    \\
$1.50$ & 512  & 4.7244e-05 & 2.07 & 9.50   & & 4.7244e-05 & 2.07 & 33.27   \\
       & 1024 & 1.1439e-05 & 2.05 & 26.64  & & 1.1439e-05 & 2.05 & 121.90  \\
       & 2048 & 2.7992e-06 & 2.03 & 84.22  & & 2.7992e-06 & 2.03 & 471.38  \\
\noalign{\smallskip}\hline\noalign{\smallskip}
       & 128  & 5.8495e-04 & *    & 2.13   & & 5.8495e-04 & *    & 3.55    \\
       & 256  & 1.4113e-04 & 2.05 & 4.67   & & 1.4113e-04 & 2.05 & 10.04   \\
$1.80$ & 512  & 3.4530e-05 & 2.03 & 11.20  & & 3.4530e-05 & 2.03 & 32.72   \\
       & 1024 & 8.5231e-06 & 2.02 & 30.00  & & 8.5231e-06 & 2.02 & 119.98  \\
       & 2048 & 2.1149e-06 & 2.01 & 90.93  & & 2.1149e-06 & 2.01 & 455.23  \\
\noalign{\smallskip}\hline
\end{tabular}
\end{table}

\begin{table}\footnotesize\centering
\caption{Discrete errors, spatial convergence rates, and CPU time (seconds) with $N=256$ and $J=4$ for Case (II).}
\label{table4}  
\begin{tabular}{cccccccccc}
\hline\noalign{\smallskip}
\multirow{2}{*}{$\alpha_{0}$} &\multirow{2}{*}{$M_H$} &\multirow{2}{*}{$M_h$} &\multicolumn{3}{c}{STG algorithm} & &  \multicolumn{3}{c}{Nonlinear algorithm} \\
\cline{4-6} \cline{8-10}
& & &   $S(\tau,h)$    & $Rate^{h}$ & CPU(s) & & $S(\tau,h)$    & $Rate^{h}$ & CPU(s) \\
\noalign{\smallskip}\hline\noalign{\smallskip}
       & 4 & 16 & 7.0393e-06 & *    & 0.44  & & 6.9966e-06 & *    & 0.67  \\
       & 8 & 32 & 4.2880e-07 & 4.04 & 1.13  & & 4.2892e-07 & 4.03 & 2.34  \\
$1.30$ & 16& 64 & 2.6680e-08 & 4.01 & 4.36  & & 2.6680e-08 & 4.01 & 9.62  \\
       & 32& 128& 1.6655e-09 & 4.00 & 18.02 & & 1.6655e-09 & 4.00 & 41.08 \\
       & 64& 256& 1.0406e-10 & 4.00 & 83.67 & & 1.0406e-10 & 4.00 & 175.62\\
\noalign{\smallskip}\hline\noalign{\smallskip}
       & 4 & 16 & 4.4501e-06 & *    & 0.40  & & 4.8103e-06 & *	 & 0.63  \\
       & 8 & 32 & 2.9483e-07 & 3.92 & 1.09  & & 2.9739e-07 & 4.02 & 2.29  \\
$1.60$ & 16& 64 & 1.8531e-08 & 3.99 & 4.38  & & 1.8532e-08 & 4.00 & 9.71  \\
       & 32& 128& 1.1574e-09 & 4.00 & 18.12 & & 1.1574e-09 & 4.00 & 41.29 \\
       & 64& 256& 7.2302e-11 & 4.00 & 83.92 & & 7.2306e-11 & 4.00 & 177.10\\
\noalign{\smallskip}\hline\noalign{\smallskip}
       & 4 & 16 & 1.8446e-04 & *    & 10.02 & & 1.8354e-04 & *    & 10.20 \\
       & 8 & 32 & 1.1263e-05 & 4.03 & 10.72 & & 1.1266e-05 & 4.03 & 11.77 \\
$1.90$ & 16& 64 & 7.0120e-07 & 4.01 & 14.03 & & 7.0121e-07 & 4.01 & 18.71 \\
       & 32& 128& 4.3781e-08 & 4.00 & 27.94 & & 4.3781e-08 & 4.00 & 48.23 \\
       & 64& 256& 2.7352e-09 & 4.00 & 94.21 & & 2.7353e-09 & 4.00 & 173.12\\
\noalign{\smallskip}\hline
\end{tabular}
\end{table}

We employ the STG compact difference scheme \eqref{coarse-scheme-1}–\eqref{fine-scheme-4} 
and the standard nonlinear compact difference scheme \eqref{standard fine} to simulate two cases of the nonlinear problem. The corresponding numerical results are summarized in Tables \ref{table1}–\ref{table4}. For Case (I), by fixing $M_h = M_H^2 = 64$, the discrete errors, temporal convergence orders, and CPU times of both methods are reported in Table \ref{table1}. The spatial accuracy and efficiency are examined in Table \ref{table2}, where the discrete errors, spatial convergence orders, and CPU times are presented with $N = 256$ and $J = 4$. Under the same parameter settings, Tables \ref{table3} and \ref{table4} display the corresponding results for Case (II). From these numerical results, we observe that:

\begin{itemize}
\item Both methods attain the expected second-order accuracy in time and fourth-order accuracy in space, 
which is in complete agreement with the theoretical analysis.

\item The STG scheme achieves higher computational efficiency than the standard nonlinear scheme, while maintaining comparable numerical accuracy.

\item Even when the derivative of the nonlinear term $f(u)$ is not globally bounded, both methods retain their accuracy and computational efficiency. This indicates that the two algorithms are, to some extent, robust with respect to the magnitude or growth rate of the nonlinearity, although this assumption is required for the theoretical analysis.

\end{itemize}


\section*{Declarations}

\noindent {\textbf{Conflict of interest}}: The authors declare no competing financial interests or personal relationships that could have influenced this work.


\vskip 2mm
\noindent {\textbf{Funding}}: This work is supported by the Scientific Research Fund Project of Yunnan Provincial Education Department (No. 2024J0642), the Yunnan Fundamental Research Projects (No. 202401AU070104), the Scientific Research Fund Project of Yunnan University of Finance and Economics (No. 2024D38), and Postdoctoral Fellowship Program of CPSF (No. GZC20240938).

\vskip 2mm
\noindent {\textbf{Data Availability}}: The datasets are available from the corresponding author upon reasonable request.


\end{document}